\documentclass[a4paper,12pt]{amsart}

\usepackage{amssymb,amsmath,amsthm,amscd}
\usepackage[all]{xy}
\usepackage{setspace}
\usepackage{mathrsfs}
\usepackage{url}
\usepackage{tikz}
\everymath{\displaystyle}
\theoremstyle{plain} 
\newtheorem{theorem}{\indent\sc Theorem}[section]
\newtheorem{lemma}[theorem]{\indent\sc Lemma}
\newtheorem{corollary}[theorem]{\indent\sc Corollary}
\newtheorem{proposition}[theorem]{\indent\sc Proposition}

\theoremstyle{definition} 
\newtheorem{definition}[theorem]{\indent\sc Definition}
\newtheorem{remark}[theorem]{\indent\sc Remark}

\newcommand{\vp}{\varphi}
\newcommand{\ve}{\varepsilon}
\newcommand{\del}{\partial}
\newcommand{\delbar}{\overline{\partial}}
\newcommand{\ddbar}{\partial\overline{\partial}}

\begin{document}

\pagestyle{plain}
\thispagestyle{plain}

\title[Pluripotential geometry on semi-positive effective divisors]
{Pluripotential geometry on semi-positive effective divisors of numerical dimension one}

\author[Takayuki KOIKE]{Takayuki KOIKE$^{1}$}
\address{ 
$^{1}$ Department of Mathematics \\
Graduate School of Science \\
Osaka Metropolitan University \\
3-3-138 Sugimoto \\
Osaka 558-8585 \\
Japan 
}
\email{tkoike@omu.ac.jp}

\renewcommand{\thefootnote}{\fnsymbol{footnote}}
\footnote[0]{ 
2020 \textit{Mathematics Subject Classification}.
Primary 32J25; Secondary 32U05, 14C20, 32D15. 
}
\footnote[0]{ 
\textit{Key words and phrases}.
Semi-positive line bundles; Numerical dimension one; 
Ueda theory; Demailly--Nemirovski--Fu--Shaw type functions; Levi-flat hypersurfaces; Hartogs extension phenomenon.
}
\renewcommand{\thefootnote}{\arabic{footnote}}

\begin{abstract}
We study the complex-analytic geometry of semi-positive holomorphic line bundles on compact Kähler manifolds. 
In one of our main results, for a $\mathbb{Q}$-effective line bundle satisfying a natural torsion-type assumption, 
we show the equivalence between semi-positivity and semi-ampleness. 
More generally, for an effective nef divisor of numerical dimension one, 
we characterize the semi-positivity of the associated line bundle 
in terms of the existence of a certain type of pseudoflat fundamental system of neighborhoods of the support. 
Furthermore, for an effective semi-positive divisor, 
we prove a dichotomy: either the divisor is the pull-back of a $\mathbb{Q}$-divisor by a fibration onto a Riemann surface, 
or the Hartogs extension phenomenon holds on the complement of its support. 
Our proof is based on a pluripotential method that has previously been used for studying the boundaries of pseudoconvex domains, 
which allows us to investigate the complex-analytic structure of neighborhoods of the support of the divisor even when the manifold is non-compact.
\end{abstract}

\maketitle


\section{Introduction}

Let $X$ be a compact Kähler manifold of complex dimension $n$ and $L$ be a holomorphic line bundle on $X$. 
The purpose of this paper is to investigate the complex-analytic geometric structure of $X$ that is governed by $L$ when it is \emph{semi-positive}, namely when it admits a $C^\infty$ Hermitian (fiber) metric whose Chern curvature form is positive semi-definite at every point of $X$. 

When $L$ is the determinant bundle $K_X$ of the holomorphic cotangent bundle of $X$ (the \emph{canonical bundle}) 
or its dual $K_X^{-1}$ (the \emph{anti-canonical bundle}), 
this type of problem has long been a guiding theme in the theory of compact Kähler manifolds. 
Indeed, 
a wide range of studies has been actively pursued toward the classification of manifolds under the assumptions that $K_X$ or $K_X^{-1}$ is semi-positive, or satisfies its numerical weak analogue, namely the condition of being \emph{nef}. 
For the canonical bundle $K_X$, such investigations are developed in connection with the \emph{abundance conjecture} (see e.g., \cite[Chapter V]{Nakayama04} and \cite[\S 19]{Demailly}; for the recent progress we refer to Remark \ref{rmk:qeffopen}, see also the introduction of \cite{La}), while for the anti-canonical bundle $K_X^{-1}$, related results have been obtained in \cite{HL} for surfaces and \cite{LMPTX} for threefolds, see also \cite{DPS94, DPS, De2015, C, CH, CCM, MWa, GKP, MW, MWWZ} for the results concerning the structure theorems for manifolds with nef $K_X^{-1}$. For the case where $K_X$ is flat (i.e., when $X$ is Calabi--Yau), we refer the reader to \cite{T} for related developments. 

On the other hand, from a broader viewpoint, one may regard these types of problems as part of a long-standing theme in complex-analytic geometry:
understanding how the plurisubharmonicity of (local) weight functions determines the geometry of the underlying manifold $X$. 
In this direction, Kodaira's embedding theorem---asserting that positivity of $L$ implies its ampleness---should be regarded as a milestone in the study of compact $X$. 
Thus, beyond the canonical or anti-canonical setting, 
it seems natural to ask what kind of complex-analytic geometric structures arise from a generally given holomorphic line bundle $L$ when it is semi-positive. 

Motivated by the recent strategy developed in the study of the abundance conjecture, 
we study how semi-positivity of $L$ governs the geometry
according to the \emph{numerical dimension}
\[
 {\rm nd}(L)
 := \max \bigl\{ k\in \{0, 1, 2, \dots, n\} \;\big|\;
   (c_1^{\mathbb{R}}(L))^{\wedge k} \not= 0 \text{ in } H^{2k}(X,\mathbb{R}) \bigr\},  
\]
where $c_1^{\mathbb{R}}(L)$ denotes the image of the first Chern class $c_1(L)$ of $L$ 
by the natural map $H^2(X, \mathbb{Z}) \to H^2(X, \mathbb{R})$ between the cohomology groups. 
Classical theorems correspond to the two extreme cases:
for ${\rm nd}(L)=n$, the situation corresponds to the characterization of the \emph{bigness} \cite{De1985} (see also \cite{DP} and \cite[(6.19)]{Demailly}), 
while for ${\rm nd}(L)=0$, a theorem of Kashiwara on the characterization of the \emph{unitary flatness} (see \cite[\S 1]{U}). 
In the context of the abundance conjecture, the cases ${\rm nd}(K_X)=0$ and ${\rm nd}(K_X)=n$ are completely settled \cite{Ka2, S, Nakayama04}, 
and the case ${\rm nd}(K_X)=1$ is currently under intensive study. 
Even under the additional assumption that $L$ admits non-zero holomorphic sections (after some tensor power), 
the problem seems to still remain unsolved (see Remark \ref{rmk:qeffopen}). 
In this paper, in view of the above circumstances on the abundance conjecture, we focus on the case where $L$ coincides with the holomorphic line bundle $[D]$ associated with an effective divisor $D$ and satisfies ${\rm nd}(L)=1$. 

In this setting, inspired by Brunella's work \cite{B} on certain rational surfaces as a model case, 
we previously investigated the complex-analytic geometric structure of a neighborhood of a (reduced) complex hypersurface $Y$ in $X$ when $L=[Y]$ in \cite{Ko2015, Ko2017, Ko2020, Ko2021, Ko2022, Ko2023, Ko2024, Ko2024delbar, Ko20242}, 
in connection with \emph{Ueda theory}~\cite{U}. 
In that framework, the problems were formulated in relation to the conjectures proposed in \cite[Conjecture 1.1]{Ko2021} and \cite[Conjecture 2.1]{Ko2023}, 
where the study was mainly based on pluripotential methods on compact Kähler manifolds and complex dynamical methods. 
When $Y$ is smooth, a positive result confirming these conjectures was obtained in \cite[Corollary 1.5]{Ko2024}. 

However, these previous results were restricted to the case where $Y$ is a smooth hypersurface or a curve with simple singularities (such as nodes) in a surface, 
as the dynamical approach based on Ueda theory is difficult to extend beyond the non-singular setting. 
In the present paper, we introduce a pluripotential-geometric approach inspired by \cite{N} and \cite{FS}, 
which allows us to treat general effective divisors. 
The following theorem presents our first main result in the present paper, which is a generalization of \cite[Corollary 1.5]{Ko2024}.

\begin{theorem}\label{thm:main_gen}
Let $X$ be a compact Kähler manifold and $D$ be an effective nef divisor on $X$ such that ${\rm nd}([D])=1$. Denote by $Y$ the support of $D$. Then the following conditions are equivalent: \\
$(i)$ The line bundle $[D]$ is semi-positive. \\
$(ii)$ There exists a neighborhood $V$ of $Y$ such that the restriction $[D]|_V$ is unitary flat. \\
$(iii)$ There exist a covering $\{V_j\}$ of $Y$ by open subsets of $X$ and holomorphic functions $s_j\colon V_j \to \mathbb{C}$ whose zero sets coincide with $V_j\cap Y$, such that $|s_j/s_k|\equiv 1$ on each intersection $V_j\cap V_k$. 
\end{theorem}

It should be noted that, when condition $(iii)$ of Theorem \ref{thm:main_gen} holds, the real hypersurfaces 
$\{x\in V_j \mid |s_j(x)| = \varepsilon\}$ 
glue together, for sufficiently small positive numbers $\varepsilon$, 
to form a compact Levi-flat hypersurface in $X\setminus Y$. 
Consequently, the support $Y$ then admits a certain type of fundamental system of pseudoflat neighborhoods, 
which can be regarded as a functional-analytic feature of the neighborhood structure of $Y$ 
characterizing the existence of a semi-positive effective divisor $D$ of numerical dimension one with support $Y$.

To clarify the connection with the abundance conjecture, 
let us now consider the situation where $D=K_X$ (or more generally, a divisor of a similar nature appearing in the context of the generalized variant of this conjecture, see \cite[(19.16)]{Demailly} for example). 
In this case, when the numerical dimension is one, 
it can be naturally expected that the restriction $[mD]|_Y$ becomes holomorphically trivial for some positive integer $m$ (see Remark \ref{rmk_torsion_normalbundle} and \S \ref{section:obs_abundance}), 
and it is conjectured that the complete linear system $|mD|$ should yield a fibration on $X$ (\emph{semi-ampleness}, see also \cite[Theorem 1.1]{Kawamata1985}). 
In view of this, as a precise variant of Theorem \ref{thm:main_gen}, we next establish an equivalence between semi-positivity and semi-ampleness 
under the above torsion-type assumption, 
which generalizes our previous result \cite[Theorem 1.1 $(i)$]{Ko2022} for the case where $D=Y$ for a non-singular hypersurface $Y\subset X$. 

\begin{theorem}\label{thm:main_tor}
Let $X$ be a connected compact Kähler manifold and $D\not=0$ be an effective divisor on $X$ such that $[mD]|_Y$ is holomorphically trivial for some positive integer $m$, where $Y$ is the support of $D$. 
Then the following conditions are equivalent: \\
$(i)$ The line bundle $[D]$ is semi-positive. \\
$(ii)$ There exist a fibration (that is, a proper holomorphic surjection with connected fibers) $\pi\colon X\to R$ onto a compact Riemann surface $R$ and an effective divisor $D_R$ such that $\pi^*D_R = m' D$ holds as divisors on $X$ for some positive integer $m'$. 
\end{theorem}

As mentioned earlier, the proofs of the above results rely on a rather different approach from our previous works: 
In the present paper, we make use of certain pluripotential-theoretically defined functions that were introduced in the study of pseudoconvex domains. 
A framework for studying level–set Monge–Ampère expressions of the form
$d^c\rho \wedge (dd^c\rho)^{n-1}$ for a plurisubharmonic function $\rho$ appears in Demailly's work \cite{Demailly1987} on boundary Monge–Ampère measures and pluriharmonic measures. 
In Nemirovski's proof of the nonexistence theorem for two-dimensional Stein domains with compact Levi-flat boundaries 
admitting plurisubharmonic defining functions \cite{N}, 
a function of the form
\begin{equation}\label{eq:Nemirovski_func_original}
F(s):=\int_{\{\rho=s\}} d^c\rho \wedge dd^c\rho,
\end{equation}
was introduced independently and plays a key role in his argument, where $\rho$ is a defining function of a Stein domain. 
Related techniques were later employed by Fu and Shaw \cite{FS} 
in their study of the Diederich--Fornæss exponent. 
In our context, for a semi-positive line bundle $[D]$, 
we consider a variant of such functions on $X\setminus Y$ (\emph{Demailly--Nemirovski--Fu--Shaw type functions}, see \S \ref{subsection:conseq_nemi_type_func}) with respect to an appropriate plurisubharmonic function, 
and we show that its level sets are Levi-flat near $Y$. 

This line of argument underlying the above two main theorems 
may be viewed as a translation of the pluripotential-theoretic ideas 
appearing in \cite{Demailly1987, N, FS} 
into the complex-algebraic-geometric setting of semi-positive line bundles. 
Beyond this conceptual shift, however, the present approach possesses a decisive technical advantage over the previous methods used in \cite{Ko2024}: it enables us to treat non-compact ambient manifolds $X$ as well. 
In particular, this provides the first pluripotential framework that can address problems such as \cite[Conjecture 1.1]{Ko2021} and \cite[Conjecture 2.1]{Ko2023} even when $X$ is non-compact. 
Indeed, the two main theorems above are both derived 
from (a variant of) the following result, in which $X$ is not assumed to be compact. 

\begin{theorem}\label{thm:main_vgeneral}
Let $X$ be a Kähler manifold of dimension $n\geq 2$ and $D$ be an effective divisor on $X$ whose support $Y$ is compact and connected. 
Assume 
that $[D]$ is semi-positive, and that 
the intersection number $(c_1([D])^2.\ \alpha^{n-2})$ is zero for a Kähler class $\alpha$ of $X$.  
Then one of the following holds: \\
$(i)$ For every neighborhood $V$ of $Y$, 
there exists a smooth one-parameter family $\{\Gamma_\tau\}_{\tau\in I_V}$, 
where the index set $I_V\subset\mathbb{R}$ is a non-empty open interval, 
of non-singular compact complex hypersurfaces in $V\setminus Y$ 
such that $\Gamma_\tau$ and $\Gamma_{\tau'}$ are $C^\infty$-homotopic to each other in $V\setminus Y$ and $\Gamma_\tau\cap\Gamma_{\tau'}=\emptyset$ for all $\tau\not=\tau'\in I_V$. \\
$(ii)$ There exists a neighborhood $V$ of $Y$ such that $[D]|_V$ is unitary flat. 
\end{theorem}

Finally, let us turn our attention to the complex-analytic aspects of the present results when $X$ is a connected compact Kähler manifold. 
From the viewpoint of several complex variables, 
divisors $D$ with only minimal (semi-)positivity assumptions 
have also been studied in relation to $q$-convexity of the complement $X\setminus Y$, where $Y$ is the support of $D$, 
and to the Hartogs extension phenomenon on $X\setminus Y$ as a consequence of $L^2$-theoretic arguments.
For instance, Ohsawa \cite{O20122} proved that 
if $[D]$ admits a $C^\infty$ Hermitian metric whose curvature form is semi-positive on the Zariski tangent spaces of $Y$ everywhere and not identically zero at some point of $Y$, then the Hartogs extension holds on $X\setminus Y$. 
More recently, Feklistov \cite{F2025} characterized the Hartogs extension phenomenon in cohomological terms in the case where $D$ is effective and nef. 
In this direction, we obtain the following complementary statement for effective semi-positive divisors of general numerical dimension. 

\begin{theorem}\label{cor:main}
Let $X$ be a connected compact Kähler manifold and $D$ be an effective divisor with connected support $Y$ such that the line bundle $[D]$ is semi-positive. Then exactly one of the following two conditions holds. \\
$(i)$ There exists a fibration $\pi\colon X\to R$ onto a compact Riemann surface $R$ such that $\pi^*\{p\} = aD$ as divisors on $X$ for some point $p\in R$ and some positive rational number $a$. \\
$(ii)$ The Hartogs type extension theorem holds on $X\setminus Y$: i.e., for any connected neighborhood $V$ of $Y$ and any holomorphic function $f$ on $V\setminus Y$, there exists a holomorphic function $\widetilde{f}$ on $X\setminus Y$ such that $\widetilde{f}|_{V\setminus Y} = f$. 
\end{theorem}

The organization of the paper is as follows.
In \S 2 we review basic definitions, notation, and known results that will be used throughout the paper.
\S 3 collects preparatory results and observations which serve as tools for the proofs of the main theorems.
\S 4 is devoted to the proofs of the main theorems.
In \S 5 we present several examples, discuss the torsion-type assumption in Theorem \ref{thm:main_tor} from the viewpoint of the abundance conjecture, and derive a result for Calabi–Yau manifolds (Corollary \ref{cor:CYleviflat}) as a further outcome of one of our main results. 
\vskip3mm
{\bf Acknowledgment. } 
This article was written while the author was visiting the Bergische Universität Wuppertal. 
He would like to thank this institution and the members of the Complex Analysis Group Wuppertal, 
especially Professors Nikolay Shcherbina, Stefan Nemirovski, and Thomas Pawlaschyk, 
for their warm-hearted hospitality and many fruitful discussions. 
The author is also deeply grateful to Professor Kazuko Matsumoto; 
indeed, this work was initiated by a discussion following her seminar talk during her stay at Wuppertal, where Professor Stefan Nemirovski pointed out the relevance of the results and techniques in \cite{N} and \cite{FS}. 
The author also thanks Professor Masanori Adachi for pointing out the relevance of Demailly's work \cite{Demailly1987} to the present context. 
He furthermore thanks Professors Yoshinori Gongyo, Vladimir Lazić, Masataka Iwai, Shin-ichi Matsumura, and Valentino Tosatti for helpful and insightful discussions,
as well as for informing him of the latest developments relevant to \S 1 and 5.
This work was supported by JSPS KAKENHI Grant Numbers JP23K03119 and JP22KK0232 
(Fund for the Promotion of Joint International Research (Fostering Joint International Research (A))).

\section{Preliminaries}

\subsection{Notation and basic conventions}
In this subsection, we fix the notation and conventions used throughout this paper.
Unless otherwise stated, any connected set will be assumed to be non-empty, any complex manifold will be assumed to be paracompact and Hausdorff, and any Riemann surface will be assumed to be connected.

For a complex manifold $\Omega$ of dimension $n$, a divisor $D$ on $\Omega$, and a holomorphic line bundle $L$ on $\Omega$, we fix our notational conventions as follows. 
\begin{itemize}
\item $\mathbb{Z}$ denotes the ring of integers.  
\item $\mathbb{Q}$, $\mathbb{R}$, and $\mathbb{C}$ denote the fields of rational, real, and complex numbers, respectively. 
\item $\mathbb{P}^\ell$ denotes the complex projective space of dimension $\ell$. 
\item ${\rm U}(1)$ denotes $\{z \in \mathbb{C} \mid |z| = 1\}$. 
\item $\mathbb{Z}_+^N$ and $\mathbb{R}_+^N$ denote the sets of all $N$-dimensional vectors 
\[
\mathfrak{m}=\begin{pmatrix}
m_1  \\
m_2  \\
\vdots \\
m_N
\end{pmatrix} 
\]
such that each $m_i$ is an integer and a real number, respectively, and $m_i > 0$ for all $i = 1, 2, \dots, N$. 
\item For a topological space $Z$, $\pi_1(Z, *)$ denotes the fundamental group of $Z$, where $*\in Z$ is a base point. 
\item $\mathcal{O}_\Omega$ denotes the sheaf of germs of holomorphic functions on $\Omega$. 
\item For a function $f\colon \Omega \to \mathbb{C}$, ${\rm Re}\,f$ and ${\rm Im}\,f$ denote the real and the imaginary parts of $f$, respectively. 
\item $\mathcal{O}_\Omega(L)$ denotes the sheaf of germs of holomorphic sections of $L$ on $\Omega$. 
\item ${\rm Supp}\,D$ denotes the support of $D$. 
\item $[D]$ denotes the holomorphic line bundle associated to $D$ (see, e.g., \cite[Chapter I, \S 1]{GH}). 
\item $T_D$ denotes the current of integration along $D$, namely $T_D$ is the closed semi-positive current defined by letting 
\[
\langle T_D, \phi \rangle := \sum_{j=1}^N m_j\cdot \int_{Y_j}\phi
\]
for $(n-1, n-1)$-form $\phi$ with compact support, where $D=\textstyle\sum_{j=1}^Nm_jY_j$ is the irreducible decomposition of $D$ (Note that $\langle T_D, \phi \rangle$ is well-defined for any $(n-1, n-1)$-form $\phi$ if the support of $D$ is compact. See, e.g., \cite[\S 1]{Demailly} and \cite[Chapter III, \S 2]{Demaillyagbook}). 
\item For an integer $m$, $L^m := L^{\otimes m}$ denotes the $|m|$-th tensor power of $L$ if $m \geq 0$, and the dual line bundle of $L$ if $m < 0$. 
\item For a $C^\infty$ Hermitian metric $h$ on $L$, $\Theta_h$ denotes the Chern curvature form of $h$. 
\item $\del$ and $\delbar$ denote the complex exterior derivatives of type $(1, 0)$ and $(0, 1)$, respectively.
\item $d^c$ denotes the operator $\textstyle\frac{1}{4\pi\sqrt{-1}}(\del - \delbar)$. 
\item For a reduced analytic space $Y$ and a sheaf $\mathcal{F}$ on $Y$, 
we denote by $\Gamma(Y, \mathcal{F})$ the set of sections of $\mathcal{F}$ on $Y$, and by $H^q(Y,\mathcal{F})$ the $q$-th cohomology group of $\mathcal{F}$ on $Y$. 
When $\mathcal{F}$ is the sheaf of locally constant functions 
with values in a group, or the sheaf of holomorphic sections of a holomorphic vector bundle, 
we naturally identify $H^q(Y,\mathcal{F})$ 
with the corresponding singular (or de Rham, if the group is $\mathbb{R}$ or $\mathbb{C}$ and $Y$ is non-singular) or Dolbeault cohomology group (if $Y$ is non-singular), respectively (For the details, refer to \cite[\S 2.2]{Ko20242}). 
\item When $\Omega$ is a compact complex manifold of dimension $n$, $\chi(\Omega, \mathcal{O}_{\Omega})$ denotes the holomorphic Euler characteristic $\textstyle\sum_{q=1}^n(-1)^q{\rm dim}\,H^q(\Omega, \mathcal{O}_\Omega)$ of $\Omega$. 
\item For a $(p, q)$-form $\alpha$ on $\Omega$, $\{\alpha\}$ denotes its class in the Dolbeault cohomology group $H^{p, q}(\Omega)$. 
\end{itemize}

Note that, according to our definition of the operator $d^c$, 
the Poincaré--Lelong formula (see, for example, \cite[Chapter III, (2.15)]{Demaillyagbook}) can be written as
\[
dd^c \log |f|^2 = T_D,
\]
where $f \in \Gamma(\Omega, \mathcal{O}_\Omega)$ and $D = {\rm div}\, f$ is the divisor of $f$.

We also adopt the following notational simplifications: 
when no confusion is likely to arise, we often write $c_1^{\mathbb{R}}$ simply as $c_1$, 
and for a divisor $D$, we write $c_1([D])$ simply as $c_1(D)$. 
Similarly, intersection numbers (in the case where $\Omega$ or ${\rm Supp}\,D$ is compact) such as 
\[
(c_1^{\mathbb{R}}([D])^k, \alpha) := \int_{\Omega} c_1^{\mathbb{R}}([D])^{\wedge k}\wedge \alpha, 
\]
where $k\in \{0, 1, \dots, n\}$ and $\alpha \in H^{n-k, n-k}(\Omega)$, will be denoted by $(D^k, \alpha)$, following standard conventions. 

%

\subsection{Standing assumptions and setting}\label{subsection:standing_seting}
In this subsection, we fix the assumptions and the geometric setting that will be frequently referred to throughout the paper. We let: 
\begin{itemize}
\item $X$ be a complex manifold of complex dimension $n \geq 1$ (since in most of the statements and arguments in this paper the case $n=1$ is obvious, we will mainly consider the case $n>1$ unless otherwise stated);
\item $Y \subset X$ be a connected compact reduced complex hypersurface, whose irreducible decomposition is denoted by 
\[
Y=\bigcup_{\nu=1}^NY_\nu; 
\]
\item $D$ be an effective divisor such that ${\rm Supp}\, D = Y$;
\item $h_D$ be a $C^\infty$ Hermitian metric on $[D]$.
\end{itemize}

We shall focus on the function
\[
  \psi := - \log |\sigma_D|_{h_D}^2,
\]
and also on 
\[
  \eta := |\sigma_D|_{h_D}^2 = e^{-\psi},
\]
where $\sigma_D\in \Gamma(X, \mathcal{O}_X([D]))$ is the canonical section of $[D]$. 
When we consider their level sets, we use the parameters $s$ and $t$ which are related by 
\begin{equation}\label{eq:relation_sandt}
s = -\log t, 
\end{equation}
and simply denote them by $\{\psi = s\}$ and $\{\eta = t\}$; 
in such cases, $s$ is regarded as sufficiently large and $t$ as sufficiently small and positive, since in most of the arguments our interest lies in small neighborhoods of $Y$.

In view of Remark \ref{rmk:psi-h-corresp} below, 
by suitably modifying the function $\psi$ outside of a small neighborhood of $Y$ and 
replacing $h_D$ with the corresponding Hermitian metric on $[D]$, 
we may always arrange the following property to hold, since one can connect a point close to $Y$ and a point of $Y$ by a path by considering, for example, the downwards flow of $\eta$ (refer to an argument in the proof of Lemma \ref{lem:psi_levelset_conn} for the non-vanishing of $d\eta$ at a point close to $Y$): 
\begin{quote}
{\bf Property $(\ast)$.}\;
For sufficiently small $t_0>0$, 
the sublevel set $\{\eta<t_0\}$ is a connected open neighborhood of $Y$. \qed 
\end{quote}
Moreover, as will be seen later in Remark \ref{rmk:wlogwmaV0conn}, 
this modification can be carried out in such a way that the semi-positivity condition $\sqrt{-1}\Theta_{h_D}\ge0$ is preserved whenever it holds. 
Throughout \S 3 and in the subsequent sections, 
we shall often assume that $h_D$ satisfies {\bf Property $(\ast)$}, 
mainly when this can be done without loss of generality. 

\begin{remark}\label{rmk:psi-h-corresp}
Fix $X$ and $D$ as above. 
As above, once a $C^\infty$ Hermitian metric $h_D$ on $[D]$ is fixed, 
the corresponding function $\psi$ is determined. 
Here we note that, conversely, a function on $X$ having logarithmic singularities only along $Y$ (such as $\psi$ defined above) determines an Hermitian metric $h_D$ on $[D]$ by letting $\varphi := \psi + \log|\sigma|^2$ be its local weight function, where $\sigma$ is a local defining function of $D$, on each locus. 
Note also that, it follows from the Poincaré--Lelong formula that they are related by
\begin{equation}\label{eq:Poincare_lelong_ddcpsi}
  dd^c \psi = - T_D + \frac{\sqrt{-1}}{2\pi} \Theta_{h_D}. 
\end{equation}
\end{remark}

\subsection{Numerical triviality along $Y$ with respect to Kähler classes of $X$}

For $X$, $Y$, and $D$ as in \S \ref{subsection:standing_seting}, we shall mainly investigate the case where $D$ is nef and the numerical dimension ${\rm nd}(D):={\rm nd}([D])$ equals $1$ when $X$ is compact Kähler. 
When $X$ is not assumed to be compact, we shall often focus on the case where $D$ enjoys the following condition, motivated by Lemma \ref{lem:num1_flatrestr} below. 
\begin{definition}\label{def:num_triv_along_Y_wrt_K}
Let $X$, $Y$, and $D$ be as above. Assume that $X$ is Kähler. 
We say that \emph{$D$ is numerically trivial along $Y$ with respect to Kähler classes of $X$} if 
\begin{equation}\label{eq:indefofnumtrivalongY}
\int_{Y_\nu} c_1(D) \wedge \alpha^{\wedge (n-2)} = 0
\end{equation}
holds for any Kähler class $\alpha$ of $X$ and for any $\nu$. 
\end{definition}

Note that, when $X$ is Kähler and $[D]$ is semi-positive (or more generally $D$ is nef, when $X$ is compact), $D$ is numerically trivial along $Y$ with respect to Kähler classes of $X$ if and only if there exists a Kähler class $\alpha$ of $X$ such that (\ref{eq:indefofnumtrivalongY}) holds, since, for any Kähler forms $\omega$ and $\omega'$ on $X$, we have $\ve\omega <  \omega' < M\omega$ at every point of $Y$ for some positive numbers $\ve$ and $M$ (Here we use the compactness of $Y$). 
For example, an effective divisor $D$ on a Kähler manifold $X$ is numerically trivial along $Y$ with respect to Kähler classes of $X$ if $c_1^{\mathbb{R}}([D]|_{Y_\nu}) = 0$ holds for every $\nu$. 

\begin{lemma}\label{lem:num1_flatrestr}
Assume that $X$ is compact Kähler and that $D\not=0$ is an effective nef divisor on $X$. 
Then $D$ is numerically trivial along $Y$ with respect to Kähler classes of $X$ if ${\rm nd}(D) = 1$. The converse also holds when $[D]$ is semi-positive. 
\end{lemma}

\begin{proof}
The first half of the assertion follows from the equation 
\begin{equation}\label{eq:keisan_D2alphan-2}
(D^2, \alpha^{n-2}) 
= \sum_{\nu=1}^N m_\nu \cdot \int_X c_1(Y_\nu) \wedge c_1(D) \wedge \alpha^{\wedge (n-2)} 
= \sum_{\nu=1}^N m_\nu \cdot \int_{Y_\nu} c_1(D) \wedge \alpha^{\wedge (n-2)} 
\end{equation}
for a Kähler class $\alpha$ of $X$, where $m_\nu$ is the coefficient of $Y_\nu$ in the irreducible decomposition of $D$ (Here we applied the Poincaré--Lelong formula by considering a canonically attached singular Hermitian metric on each $[Y_\nu]$, see \cite[(3.13)]{Demailly} for further details). 
To show the converse, under the assumption that $[D]$ admits a $C^\infty$ Hermitian metric $h_D$ with $\sqrt{-1}\Theta_{h_D}\geq 0$, assume that $D$ is numerically trivial along $Y$ with respect to Kähler classes of $X$. 
Then, again by using equation (\ref{eq:keisan_D2alphan-2}), we have that $(D^2. \{\omega\}^{n-2})=0$ holds for a Kähler form $\omega$ of $X$. Therefore, for the eigenvalues $0\leq \gamma_1\leq \gamma_2\leq \dots\leq \gamma_n$ of $\sqrt{-1}\Theta_{h_D}$ with respect to $\omega$, we have 
\[
\sum_{1\leq j<k\leq n} \int_X \gamma_j\gamma_k\cdot \omega^{\wedge n}
 = c_n\cdot\int_X (\sqrt{-1}\Theta_{h_D})^{\wedge 2}\wedge \omega^{\wedge (n-2)} = 0
\]
holds, where $c_n$ is a positive constant that depends only on $n$. 
From this equation, it follows that $\gamma_j\gamma_k=0$ holds at any point of $X$ for any $1\leq j<k\leq n$. Thus we have $(\sqrt{-1}\Theta_{h_D})^{\wedge 2}\equiv 0$ as $(2, 2)$-forms, from which the assertion follows. 
\end{proof}

\begin{remark}\label{rmk_torsion_normalbundle}
Here let us consider the case where $D=Y$ for a non-singular complex hypersurface $Y$ of a compact Kähler manifold $X$. 
If $D$ is nef and ${\rm nd}(D)=1$, then by Lemma \ref{lem:num1_flatrestr} we have $[D]|_Y = [Y]|_Y$, which can be naturally identified with the (holomorphic) normal bundle $N_{Y/X}$ of $Y$, satisfies $(c_1([D]|_Y). \alpha_Y^{n-2}) = (c_1([D]|_Y)^2. \alpha_Y^{n-3}) = 0$ for a Kähler class $\alpha_Y$ of $Y$. 
Therefore it follows from \cite[Theorem A]{DN} that $c_1^{\mathbb{R}}([D]|_Y)=0$ (see also the argument in the end of \cite[Remark 3.4]{TZ}). 
Thus, by Kashiwara's theorem, $N_{Y/X}$ is unitary flat. 
Such flat line bundles correspond to elements of $H^1(Y,{\rm U}(1))$ (see Lemma \ref{lem:flatlb_funds} $(i)$), 
and the assumption in Theorem \ref{thm:main_tor} means that the element corresponding to $N_{Y/X}$ is torsion in this group. 
\end{remark}

\subsection{Some linear-algebraic facts from \cite{KU}}\label{subsection:linalg}

For the reader's convenience, we recall in this subsection 
some linear-algebraic facts established in \cite{KU}, 
restricting ourselves to those that will be used in the present paper.

Let $A$ be a square matrix of order $N\geq 1$ which satisfies the following three properties. 
\begin{description}
\item[(P$_1$)] $A$ is a symmetric matrix with real entries. 
\item[(P$_2$)] Any non-diagonal entry of $A$ is non-negative. 
\item[(P$_3$)] For any partition $(I, J)$ of $\{1, 2, \dots, N\}$, there exist elements $i\in I$ and $j\in J$ such that the $(i, j)$ entry of $A$ is positive. 
\end{description}
Here we say that $(I, J)$ is a {\it partition} of $\{1, 2, \dots, N\}$ 
if $I$ and $J$ are non-empty subsets of $\{1, 2, \dots, N\}$ such that $I\cap J=\emptyset$ and $I\cup J=\{1, 2, \dots, N\}$ hold. 

\begin{proposition}[={\cite[Proposition 3.2]{KU}}]\label{prop:lin_alg_main}
Let $N$ be a positive integer and $A$ be a square matrix of order $N$ which satisfies {\bf (P$_1$)}, {\bf (P$_2$)}, and {\bf (P$_3$)}. Then the following holds for the largest eigenvalue $\lambda(A)$ of $A$. \\
$(i)$ $\lambda(A)>0$ if and only if there exists an element $\mathfrak{m}\in \mathbb{R}_+^N$ such that $A\mathfrak{m}\in \mathbb{R}_+^N$. \\
$(ii)$ $\lambda(A)=0$ if and only if there exists an element $\mathfrak{m}\in \mathbb{R}_+^N$ such that $A\mathfrak{m}=0$. \\
$(iii)$ $\lambda(A)<0$ if and only if there exists an element $\mathfrak{m}\in \mathbb{R}_+^N$ such that $-A\mathfrak{m}\in \mathbb{R}_+^N$. \\
In each of the assertions $(i)$, $(ii)$, and $(iii)$, ``\,$\mathfrak{m}\in \mathbb{R}_+^N$'' can be replaced with ``$\,\mathfrak{m}\in \mathbb{Z}_+^N$'' when all entries of $A$ are integers. 
\end{proposition}

In the proof of Proposition \ref{prop:lin_alg_main}, we used the following lemma. 
\begin{lemma}[={\cite[Lemma 3.3]{KU}}]\label{lem:lin_alg}
Let $N$ be a positive integer, $A$ a square matrix of order $N$ which satisfies {\bf (P$_1$)}, {\bf (P$_2$)}, and {\bf (P$_3$)}, 
and $\mathfrak{m}\in \mathbb{R}^N$ be an eigenvector associated to the largest eigenvalue of $A$. 
Then either $\mathfrak{m}$ or $-\mathfrak{m}$ is an element of $\mathbb{R}_+^N$. 
\end{lemma}

\subsection{Notions of (semi-)positivity and unitary-flatness of line bundles}\label{fundamentals_on_sp}

In this subsection, we briefly recall several standard notions and facts on positivity, semi-positivity, and unitary-flatness of holomorphic line bundles. 
For the other fundamental notions on the positivity of holomorphic line bundles on compact Kähler manifolds such as \emph{ampleness} and \emph{nefness}, 
we refer to \cite[Chapter 1]{GH}, \cite{DP}, and \cite[\S 6.C]{Demailly}. 
Note that we say that a holomorphic line bundle $L$ on a complex manifold $X$ is 
\emph{semi-ample} if there exists a positive integer $m$ such that the sheaf $\mathcal{O}_X(L^m)$ is generated by global sections. 

A holomorphic line bundle $L$ on a complex manifold $X$ is said to be \emph{positive} (resp. \emph{semi-positive}) if it admits a $C^\infty$ Hermitian metric $h$ 
such that $\sqrt{-1}\Theta_h > 0$ (resp. $\sqrt{-1}\Theta_h \geq 0$) holds at every point of $X$. By using the notion of \emph{local weight functions} of an Hermitian metric (see \cite[\S 2.1]{Ko2021} for the definition), $L$ is positive (resp. semi-positive) if and only if it admits an Hermitian metric whose local weight functions are $C^\infty$ strongly plurisubharmonic (resp. plurisubharmonic). 
For basic facts and properties on semi-positive line bundles, see also \cite[\S 1.1]{Ko2021}, \cite[\S 2.1]{Ko2023}, and \cite[\S 6]{Ko2024}. 

A complex line bundle on a CW complex (e.g., on a reduced analytic subspace of a complex manifold \cite{L}) 
is said to be \emph{unitary flat} if it admits locally constant transition functions valued in $\mathrm{U}(1)$ 
for some choice of local trivializations. 
Throughout this paper, on a holomorphic line bundle, by ``a unitary flat structure'' we always mean one compatible with its holomorphic structure. 
Note that, by definition, any unitary flat line bundle on a complex manifold has the natural structure as a holomorphic line bundle, and admits a \emph{flat metric}, namely an Hermitian metric whose local weight functions are pluriharmonic. 

\begin{lemma}\label{lem:flatlb_funds}
For a holomorphic line bundle $L\to Z$ on a reduced connected analytic subspace $Z$ of a complex manifold, the following holds: \\
$(i)$ The set of isomorphism classes of unitary flat line bundles on $Z$ can be naturally identified with $H^1(Z, {\rm U}(1))\;(\cong \mathrm{Hom}(\pi_1(Z, *), {\rm U}(1)))$. \\
$(ii)$ When $Z$ is compact, a unitary flat structure on $L$, if it exists, is unique.\\
$(iii)$ When $Z$ is non-singular, $L$ admits a flat metric if and only if $L$ carries a unitary flat structure. 
\end{lemma}

\begin{proof}
For $(i)$, 
we first note that, by the universal coefficient theorem and the fact that ${\rm U}(1)$ is abelian and divisible, we have natural isomorphisms
\[
H^1(Z, {\rm U}(1)) \cong {\rm Hom}(H_1(Z, \mathbb{Z}), {\rm U}(1))
\cong {\rm Hom}(\pi_1(Z, *)^{\rm ab}, {\rm U}(1))
\cong {\rm Hom}(\pi_1(Z, *), {\rm U}(1)), 
\]
where $\pi_1(Z, *)^{\rm ab}$ is the abelianization of $\pi_1(Z, *)$. 
The correspondence between the set of isomorphism classes of unitary flat line bundles on $Z$ and the set $\mathrm{Hom}(\pi_1(Z, *), {\rm U}(1))$ of ${\rm U}(1)$-representations of $\pi_1(Z, *)$ can be obtained in the same manner as in the manifold case, since $Z$ admits the homotopy type of a CW complex by \L ojasiewicz's theorem \cite{L}. 
Indeed, given a unitary flat line bundle on $Z$, its monodromy representation yields an element of ${\rm Hom}(\pi_1(Z, *), {\rm U}(1))$; 
conversely, for a ${\rm U}(1)$-representation of $\pi_1(Z, *)$, the associated bundle $(\widetilde{Z}\times \mathbb{C})/\pi_1(Z, *)$ provides the corresponding unitary flat line bundle, where $\widetilde{Z}$ is the universal covering space of $Z$. 
We refer to the proof of \cite[Proposition 2.2]{Ko2020} for further details. 
For $(ii)$, refer to \cite[Lemma 2.3]{Ko20242}. 
To show $(iii)$, it is sufficient to construct a structure of a unitary flat line bundle on $L$ by assuming that $L$ admits a flat metric $h$. 
Fix an open covering $\{U_j\}$ of $Z$ and a trivialization $e_j\in \Gamma(U_j, \mathcal{O}_Z(L))$ of $L$ on each $U_j$. 
Denote by $\vp_j$ the local weight function of $h$ on $U_j$. 
By replacing $\{U_j\}$ with its refinement if necessary, one can take a holomorphic function $f_j$ on each $U_j$ such that ${\rm Re}\,f_j = \vp_j$, since each $\vp_j$ is pluriharmonic. 
Define another local trivialization $\widehat{e}_j\in \Gamma(U_j, \mathcal{O}_Z(L))$ by letting $\widehat{e}_j := \exp(f_j/2)\cdot e_j$. Then, as 
\[
|\widehat{e}_j|_h^2=
|\exp(f_j/2)\cdot e_j|_h^2 =\exp({\rm Re}\,f_j)\cdot e^{-\vp_j} = 1,
\]
it follows that the transition functions of $\widehat{e}_j$'s are locally constant with values in ${\rm U}(1)$. 
\end{proof}

Let us return to the setting of \S \ref{subsection:standing_seting}. 
In this setting, first we note that the semi-positivity of $[D]$ 
is a semi-local property that can be checked near $Y$ in the sense of the following lemma (see also \cite[\S 2.1]{Ko2021} for the case $D=Y$ for a non-singular hypersurface $Y\subset X$). 

\begin{lemma}\label{lem:sp_local}
Let $X, Y$, and $D$ be as in \S \ref{subsection:standing_seting}. 
Then the following conditions are equivalent:\\
$(i)$ The line bundle $[D]$ is semi-positive.\\
$(ii)$ There exists a neighborhood $V$ of $Y$ such that the restriction $[D]|_V$ is semi-positive.
\end{lemma}

\begin{proof}
The implication $(i) \Rightarrow (ii)$ is obvious. 
The converse can be shown by a \emph{regularized minimum construction}, 
as in \cite[\S 2.1]{Ko2023}. 
For the sake of completeness, we include here a brief proof. 
Assume that there exists a neighborhood $V$ of $Y$ and a 
$C^\infty$ Hermitian metric $h_V$ on $[D]|_V$ 
such that $\sqrt{-1}\Theta_{h_V} \ge 0$. 
Let $\psi_V := -\log |\sigma_D|_{h_V}^2$, where $\sigma_D$ is the canonical section of $[D]$. 
By equation $(\ref{eq:Poincare_lelong_ddcpsi})$, $\psi_V$ is plurisubharmonic on $V \setminus Y$. 

Take a sufficiently small connected open neighborhood $W \Subset V$ of $Y$. 
Since $Y$ is compact, $\partial W$ is compact as well, 
and we can set $M := \textstyle\max_{\partial W} \psi_V < +\infty$. 
Let $\chi \colon \mathbb{R} \to \mathbb{R}$ be a $C^\infty$ function 
that is nondecreasing, convex, 
satisfies $\chi'(t)=1$ for $t \gg 1$, 
and is constant, say $\lambda$, for $t < 2M$. 
Define a function $\psi$ on $X$ by
\[
  \psi(x) := 
  \begin{cases}
    \chi\circ \psi_V(x) & \text{if } x \in W,\\[4pt]
    \lambda & \text{otherwise.}
  \end{cases}
\]
The function $\psi$ is continuous and plurisubharmonic on $X\setminus Y$, 
and it has logarithmic singularities along $Y$. 
Hence, by Remark \ref{rmk:psi-h-corresp}, it determines a 
$C^\infty$ Hermitian metric $h_D$ on $[D]$ (Note that $\psi - \psi_V$ is constant in a neighborhood of $Y$ by construction). 
As $\sqrt{-1}\Theta_{h_D} \ge 0$ on $X$, it follows that $[D]$ is semi-positive.
\end{proof}

\begin{remark}\label{rmk:wlogwmaV0conn}
In the setting of \S\ref{subsection:standing_seting}, if $[D]$ is semi-positive, 
then after modifying $h_D$ as in the proof of Lemma \ref{lem:sp_local} we may assume 
$\sqrt{-1}\Theta_{h_D}\ge0$ and that the associated function 
$\psi=-\log|\sigma_D|_{h_D}^2$ enjoys the following property:
there exists $s_0>0$ such that 
$X\setminus\{\psi\le s_0\}=\{\eta<e^{-s_0}\}$ is a connected open neighborhood of $Y$. 
\end{remark}

\begin{remark}\label{rmk:wlogwmaYconn}
Let $X$ be a complex manifold and $D\not=0$ be an effective divisor. 
When the support $Y$ of $D$ is not connected, one has the natural decomposition
\[
D = D_1 + D_2 + \cdots + D_M
\] 
such that each $D_\mu$ is an effective divisor of $X$ with connected support and that ${\rm Supp}\,D_\mu\cap {\rm Supp}\,D_{\mu'} = \emptyset$ holds for any $\mu, \mu' \in \{1, 2, \dots, M\}$ with $\mu\not=\mu'$. In this case, we have that $[D]$ is semi-positive if and only if $[D_\mu]$ is semi-positive for any $\mu$. 
Indeed, when $[D_\mu]$ admits an Hermitian metric $h_\mu$ with $\sqrt{-1}\Theta_{h_\mu} \geq 0$ for all $\mu$, by considering the Hermitian metric $h_1\otimes h_2\otimes \cdots \otimes h_M$ on $[D]$, we have that $[D]$ is also semi-positive. 
Conversely, when $[D]$ admits an Hermitian metric $h_D$ with $\sqrt{-1}\Theta_{h_D} \geq 0$, by the same argument as in the proof of Lemma \ref{lem:sp_local}, one can construct a $C^\infty$ plurisubharmonic function $\psi_\mu$ on $X\setminus {\rm Supp}\,D_\mu$ such that $\psi_\mu + \log |\sigma_D|_{h_D}^2$ is constant on a neighborhood of ${\rm Supp}\,D_\mu$ for each $\mu$, where $\sigma_D \in \Gamma(X, \mathcal{O}_X([D]))$ is the canonical section. By considering the Hermitian metric $h_\mu$ that corresponds to $\psi_\mu$ in the correspondence as in Remark \ref{rmk:psi-h-corresp}, we have that each $[D_\mu]$ is semi-positive (see also the argument in the proof of \cite[Lemma 2.5]{Ko2024delbar}). 
\end{remark}

In the setting of \S \ref{subsection:standing_seting}, for investigating the pseudoconvexity of a system of neighborhoods of $Y$, 
we will study the \emph{Levi form} of the level sets of $\psi$ (see \cite[Chapter I, \S7.C]{Demaillyagbook} for the definition and fundamental properties) for a suitable choice of $h_D$. 
In particular, in \S \ref{subsection:semilocal_nbhd_Y_Dflat}, 
we shall consider the case where $[D]|_V$ is unitary flat on some neighborhood $V$ of $Y$, since this condition appears in Theorem \ref{thm:main_gen} as condition $(ii)$. 
As in this case we may assume $dd^c\psi = -T_D$ holds on a neighborhood of $Y$ (refer to equation $(\ref{eq:Poincare_lelong_ddcpsi})$), 
each level set $\{\psi = s\}$ is a \emph{Levi-flat hypersurface} (In the present paper, a ``Levi-flat hypersurface'' means a $C^\infty$ Levi-flat hypersurface, unless otherwise stated). 
Note that the associated Levi foliation clearly coincides with the restriction of the holomorphic foliation that locally arises from the meromorphic $1$-form $\partial\psi$. 

\begin{remark}
The phenomenon that $Y$ admits such a system of pseudoflat neighborhoods 
was already observed by Brunella \cite{B} 
when $X$ is the blow-up of $\mathbb{P}^2$ at generic nine points 
in the case where its anti-canonical divisor is represented by a smooth elliptic curve $Y\subset X$, 
as a characterization of the semi-positivity of $K_X^{-1}=[Y]$. 
More generally, for compact Riemann surfaces $Y$ embedded in complex surfaces $X$ 
with topologically trivial normal bundles, 
Ueda \cite{U} provided a complex-analytic classification of neighborhoods of $Y$. 
In his classification, the classes $(\beta')$ and $(\beta'')$ 
correspond to the case where $Y$ admits pseudoflat neighborhoods described above. 
For a sufficient condition for such a pair $(Y, X)$ to be of classes $(\beta')$ and $(\beta'')$, 
Ueda also showed in \cite{U} that $Y$ admits such a pseudoflat neighborhood system 
if certain cohomological obstructions defined by considering the linearizability of transitions of local defining functions of $Y$ in each jet level along $Y$ (\emph{the Ueda classes}) vanish under a torsion- or Diophantine-type assumption on the holonomy of the normal bundle. In the direction of the generalization of his result, developments can be found in \cite{Ko2017,Ko2020,GS1}. Concerning the determination of the complex structures of such neighborhoods in similar configurations, refer to \cite{A,GS2,GS1,Ko20222,Og,KUeh,KS,SW}. 
The present paper can be said to be in line with the more pluripotential-theoretic and differential-geometric approach 
to the existence problem of pseudoflat neighborhood systems developed in 
\cite{Ko2021,Ko2022,Ko2024}, 
rather than with the dynamical approach in the works we mentioned here.
\end{remark}


\section{Preparatory results and observations for the proofs of the main theorems}
In this section we collect several preparatory results and observations 
that will be used in the proofs of the main theorems in 
\S\ref{section:prf_mainresults}. 
Most of them, especially those in 
\S\ref{subsection:conseq_linalg} and \S\ref{subsection:semilocal_nbhd_Y_Dflat}, 
are direct consequences of known results reviewed in 
\S\ref{subsection:linalg} and \S\ref{fundamentals_on_sp}, respectively, 
but we present them here in a form adapted to the setting fixed in 
\S\ref{subsection:standing_seting}. 
In \S\ref{subsection:conseq_nemi_type_func}, 
we give the definition of our 
Demailly--Nemirovski--Fu--Shaw type functions, 
which is tailored to our setting. 
In that subsection we also study some of their properties,
among which the limiting behavior plays a crucial role 
in the proofs of the main results.

\subsection{Consequences from linear–algebraic arguments}\label{subsection:conseq_linalg}

Let $X$, $Y$, and $D$ be as in \S \ref{subsection:standing_seting}.  
In this subsection we collect some consequences from linear–algebraic facts in \S \ref{subsection:linalg} which hold when $X$ is Kähler. 
We fix a Kähler form $\omega$ on $X$.  
Define an $N\times N$ matrix $A=(A_{\nu\mu})$ by
\[
A_{\nu\mu}
 := \int_X c_1(Y_\nu)\wedge c_1(Y_\mu)\wedge\{\omega\}^{\wedge(n-2)}
  = \int_{Y_\nu} c_1(Y_\mu)\wedge\{\omega\}^{\wedge(n-2)}. 
\]
Here note that the right-hand side shows that $A_{\nu\mu}$ is well-defined
even if $X$ is non-compact, because $Y$ is compact. 
Then $A$ satisfies {\bf(P$_1$)}, {\bf(P$_2$)}, and {\bf(P$_3$)} in \S \ref{subsection:linalg}: 
{\bf(P$_1$)} and {\bf(P$_2$)} follow directly from the construction,
and {\bf(P$_3$)} follows from the connectedness of $Y$. 
We can thus apply Proposition \ref{prop:lin_alg_main} and Lemma \ref{lem:lin_alg} to obtain the following: 
\begin{lemma}\label{lem:from_linalg}
Let $X, Y$, and $D$ be as in \S \ref{subsection:standing_seting}. 
Assume that $X$ is Kähler, and fix its Kähler form $\omega$ on $X$. 
Let $A$ be the matrix defined above and $D'$ be an effective divisor on $X$ whose support coincides with $Y$. 
Then the following assertions hold. \\
$(i)$ If $D$ is numerically trivial along $Y$ with respect to Kähler classes of $X$, 
then the largest eigenvalue of $A$ is $0$. \\
$(ii)$ If both $D$ and $D'$ are numerically trivial along $Y$ with respect to Kähler classes of $X$, 
then there exist positive integers $a$ and $b$ such that $aD=bD'$.
\end{lemma}

\begin{proof}
We first prove $(i)$. 
Write
\[
D=\sum_{\nu=1}^N m_\nu Y_\nu\qquad (m_\nu\in\mathbb Z_{>0}),
\]
and set 
\[
\mathfrak{m} := \begin{pmatrix}
m_1 \\ m_2 \\ \vdots \\ m_N
\end{pmatrix} \in\mathbb{Z}_+^N. 
\]
Then we have 
\[
(A\mathfrak m)_\nu
 = \int_{Y_\nu} c_1(D)\wedge
   \{\omega\}^{\wedge(n-2)}
\]
for the $\nu$-th component $(A\mathfrak m)_\nu$ of $A\mathfrak m$. 
Since $D$ is numerically trivial along $Y$ with respect to Kähler classes of $X$,
we have $A\mathfrak m=0$.
Hence, by Proposition \ref{prop:lin_alg_main} $(ii)$, 
the largest eigenvalue of~$A$ is~$0$. 

Next we prove $(ii)$.
Write
\[
D'=\sum_{\nu=1}^N m_\nu' Y_\nu\qquad (m_\nu'\in\mathbb Z_{>0}),
\]
and define a vector $\mathfrak n\in\mathbb R^N$ by
\[
\mathfrak n :=
 m_1'\!\begin{pmatrix}m_1\\ m_2\\ \vdots\\ m_N\end{pmatrix}
 -m_1 \!\begin{pmatrix}m_1'\\ m_2'\\ \vdots\\ m_N'\end{pmatrix}.
\]
Then the $\nu$-th component $(A\mathfrak n)_\nu$ of $A\mathfrak n$ is
\[
(A\mathfrak n)_\nu
 = m_1' \!\int_{Y_\nu} c_1(D)\wedge
     \{\omega\}^{\wedge(n-2)}
   - m_1 \!\int_{Y_\nu} c_1(D')\wedge
     \{\omega\}^{\wedge(n-2)}.
\]
By the assumption that both $D$ and $D'$ are numerically trivial along $Y$ with respect to Kähler classes of $X$,
we have $A\mathfrak n=0$. 
Hence $\mathfrak n$ is an eigenvector of $A$ corresponding to the
largest eigenvalue $0$ (by assertion $(i)$) if $\mathfrak n\ne0$.
In this case, by Lemma \ref{lem:lin_alg}, either $\mathfrak n$ or $-\mathfrak n$ belongs to $\mathbb R_+^N$. 
However, the first component of $\mathfrak n$ is $0$ by construction, which gives a contradiction. 
Thus $\mathfrak n=0$, and consequently we have $aD=bD'$ by letting $a=m_1'$ and $b=m_1$.
\end{proof}

\subsection{Geometry and dynamics on a neighborhood of $Y$}\label{subsection:semilocal_nbhd_Y_Dflat}
Let us again denote by $X$, $Y$, $D$, $h_D$, $\psi$, and $\eta$ 
the objects introduced in \S \ref{subsection:standing_seting}.  
In this subsection, we investigate the semi-local geometric structure 
of a system of neighborhoods of $Y$ in $X$, 
mainly in the case where the associated line bundle $[D]$ is unitary flat 
on a neighborhood of $Y$. 

We begin with the following fundamental lemma, 
for which no assumption is made on the curvature of $h_D$. 

\begin{lemma}\label{lem:psi_levelset_conn}
Let $X$, $Y$, $\psi$, and $\eta$ be as in \S\ref{subsection:standing_seting}, 
and assume that {\bf Property $(\ast)$} holds. 
Then there exists a positive number $\varepsilon_0$ such that, 
for the open neighborhood $V_0:=\{\eta<\varepsilon_0\}$ of $Y$, 
every point of $V_0\setminus Y$ is a regular point of $\psi$, 
and each level set $\{\eta=t\}$ is a connected compact real hypersurface 
(smooth submanifold of real codimension one) of $V_0\setminus Y$ 
for all $t\in(0,\varepsilon_0)$.
\end{lemma}

\begin{proof}
Locally around each point of $Y$, 
choose a holomorphic defining function $\sigma$ of $D$. 
With the corresponding local weight function $\varphi$ of $h_D$, we can write
\[
  \psi = -\log|\sigma|^2 + \varphi .
\]
Then, for any fixed Hermitian metric $g$ on $X$, we have
\[
  |d\psi|_g^2 = 2\,|\partial\psi|_g^2 
   = 2\Bigl|\frac{\partial\sigma}{\sigma} + \partial\varphi\Bigr|_g^2 .
\]
Since the first term inside the norm 
on the right-hand side has a pole along $Y$, 
it follows that all sufficiently large real numbers $s \gg 1$ 
are regular values of $\psi$ (It is clear when $Y$ is simple normal crossing. In the general case, one can apply Hironaka's log resolution 
to reduce to that situation,
or invoke the Łojasiewicz gradient inequality \cite[num\'ero 18 Proposition 1]{L2} to obtain the same conclusion). 
As $d\eta = -e^{-\psi}d\psi$, 
we can therefore find $\varepsilon_0>0$ such that 
every point of $\{\eta<\varepsilon_0\}\setminus Y$ is a regular point of $\psi$.
If necessary, we shrink $\varepsilon_0$ further so that, 
by {\bf Property $(\ast)$}, 
$V_0 = \{\eta<\varepsilon_0\}$ is connected. 

We may also assume, by shrinking $\varepsilon_0$ again if necessary, 
that $V_0$ is relatively compact in $X$, 
using the compactness of $Y$. 
Then a simple topological argument shows that 
the restriction $\eta:V_0\setminus Y\to(0,\varepsilon_0)$ is proper. 

Consequently, for each $t\in(0,\varepsilon_0)$, 
the level set $\{\eta=t\}$ is a compact real hypersurface 
of $V_0\setminus Y$.  
Finally, we prove its connectedness. 
Since $V_0$ is connected and $Y$ is a proper analytic subset of $V_0$, 
the complement $V_0\setminus Y$ is also connected 
(see \cite[Chapter~I~\S5.F and Chapter~II~(4.2)]{Demaillyagbook}). 
As the restriction $\eta:V_0\setminus Y\to(0,\varepsilon_0)$ 
is a $C^\infty$ proper submersion, 
Ehresmann's fibration theorem implies that 
$V_0\setminus Y$ is diffeomorphic to 
$\{\eta=t\}\times(0,\varepsilon_0)$ 
for any $t\in(0,\varepsilon_0)$. 
Hence each $\{\eta=t\}$ is connected, as required.
\end{proof}

We now continue the discussion under the setting of 
\S\ref{subsection:standing_seting}, 
and further assume that there exists a neighborhood $V$ of $Y$
on which the line bundle $[D]|_V$ is unitary flat.
In what follows in this subsection, let us investigate the geometry of a neighborhood of $Y$ in this case.

As mentioned in \S\ref{fundamentals_on_sp}, 
such a unitary flat structure determines a holomorphic foliation 
$\mathcal{F}$ on $V$ by letting $\{\sigma=\text{constant}\}$ be the local defining functions of its leaves, 
where $\sigma$ is a local defining function of $D$ 
that is a local frame of $[D]|_V$ 
as a unitary flat line bundle. 
Also recall that, by definition, each leaf of $\mathcal F$ other than $Y$ coincides with a 
leaf of the Levi foliation of a level set $\{\eta=t\}$ for some $t$.

For sufficiently small $t>0$, 
the intersection of $\{\eta=t\}$ with a transversal of $Y$ at a regular point $*\in Y$ is clearly a circle. 
Extending a holomorphic coordinate on this transversal
leafwise constantly along $\mathcal F$ 
and considering the monodromy of this extension along loops in $Y$,
we obtain the \emph{holonomy of $\mathcal F$ along $Y$}
\[
{\rm Hol}_{\mathcal F,Y}\colon 
\pi_1(Y,*)\longrightarrow{\rm Aut}(\Delta_t,0),
\]
where ${\rm Aut}(\Delta_t,0)$ denotes the group of holomorphic automorphisms
of the disc $\Delta_t:=\{w\in\mathbb C\mid |w|<t\}$ fixing the origin.
As is well known, ${\rm Aut}(\Delta_t,0)$ can be identified with the rotation group ${\rm U}(1)$,
and hence we may regard ${\rm Hol}_{\mathcal F,Y}$ as a representation
with values in ${\rm U}(1)$.

On the other hand, since $[D]|_Y$ is also unitary flat, 
one obtains another unitary representation of $\pi_1(Y, *)$ by considering its monodromy
\[
\rho_{D|_Y}\colon \pi_1(Y,*)\longrightarrow{\rm U}(1).
\]

\begin{remark}\label{rmk:inj_flatlb_hollb}
The representation $\rho_{D|_Y}$ is well defined by 
Lemma \ref{lem:flatlb_funds} $(i)$, $(ii)$.
\end{remark}

\begin{lemma}\label{lem:monodromy_corresp}
Under the setting of \S\ref{subsection:standing_seting}, 
assume that there exists a neighborhood $V$ of $Y$
on which $[D]|_V$ is unitary flat. 
Then, under the above identification 
${\rm Aut}(\Delta_t,0)\cong{\rm U}(1)$,
the two monodromy representations 
${\rm Hol}_{\mathcal F,Y}$ and $\rho_{D|_Y}$ coincide.
\end{lemma}

\begin{proof}
Take a neighborhood $W\subset V$ of $Y$ 
such that $Y$ is a deformation retract of $W$ (Here we applied the existence of a CW complex structure on the pair $(Y, V)$ \cite{L}). 
Then the inclusion induces an isomorphism 
$\pi_1(Y, *)\cong\pi_1(W, *)$, 
and via this identification, 
$\rho_{D|_Y}$ coincides with the monodromy
\[
\rho_{D|_W}\colon \pi_1(W,*)\longrightarrow{\rm U}(1)
\]
of the flat line bundle $[D]|_W$.
It therefore suffices to show that $\rho_{D|_W}$ 
and ${\rm Hol}_{\mathcal F,Y}$ coincide.
This follows directly from the fact that local defining functions of $D$ 
can be chosen to be $\mathcal F$-leafwise constant 
and to provide local trivializations of $[D]|_W$ 
as a unitary flat line bundle;
their transition functions along loops in $Y$
describe both monodromies in exactly the same manner.
\end{proof}

Under the same assumption that $[D]|_V$ is unitary flat, the next lemma shows that the complex-analytic properties of neighborhoods of $Y$ are essentially determined by the dynamical behavior of the above monodromy representations, at least when $X$ is Kähler. 

\begin{lemma}\label{lem:scv_monodromy_onV}
Let the setting be as in \S\ref{subsection:standing_seting}, 
and assume that there exists a neighborhood $V$ of $Y$
on which $[D]|_V$ is unitary flat. Then the following holds: \\
$(i)$ If $X$ is Kähler and the image of $\rho_{D|_Y}$ is a finite subgroup of ${\rm U}(1)$,
then, by shrinking $V$ if necessary, 
there exist a positive rational number $a$
and a proper surjective holomorphic map 
$p\colon V\to\Delta$
onto a disc $\Delta\subset\mathbb C$ centered at the origin
such that $p^*\{0\}=aD$.
In particular, every neighborhood of $Y$ contains infinitely many 
smooth compact complex hypersurfaces 
that are the supports of divisors linearly equivalent to $aD$ on $V$. \\
$(ii)$ If the image of $\rho_{D|_Y}$ is an infinite subgroup of ${\rm U}(1)$,
then, by shrinking $V$ if necessary, 
there exists no compact complex hypersurface contained in $V\setminus Y$,
and every holomorphic function on $V\setminus Y$ is constant.
\end{lemma}

\begin{proof}
$(i)$ Since the image of $\rho_{D|_Y}$ is finite, 
there exists a positive integer $\mu$ such that 
$\rho_{\mu D|_Y}=(\rho_{D|_Y})^\mu$ is the trivial representation. 
Take a small coordinate ball $B$ around a regular point of $Y$, 
and let $w$ be a local holomorphic defining function of $Y$ on $B$.
For some positive integer $m$, $w^m$ serves as a local defining function of $D|_B$. 
Define $\sigma:=w^{m\mu}$ and extend it 
leafwise constantly along the foliation $\mathcal F$.
Because $\rho_{\mu D|_Y}$ is trivial 
and by Lemma \ref{lem:monodromy_corresp},
this extension is single-valued and defines a holomorphic function 
$p$ on a neighborhood $W\subset V$ of $Y$.
Set $D':=p^*\{0\}$.
By construction, $\mathrm{Supp}\,D'=Y$, 
and for each irreducible component $Y_\nu$ of $Y$, 
the restriction $[D']|_{Y_\nu}$ is holomorphically trivial. 
Hence, by Lemma \ref{lem:from_linalg}\,(ii),
there exists a positive rational number $a$ such that $D'=aD$.
By shrinking $W$ if necessary, the map $p$ is proper and surjective,
as follows from the same topological argument as in \cite[Lemma 2.2]{KUeh}.
Thus assertion $(i)$ is proved.

$(ii)$ In this case, by Lemma \ref{lem:monodromy_corresp} 
and the relation between $\mathcal F$ and the Levi foliations
of the level sets $\{\psi=s\}$ mentioned above, it follows that each leaf of the Levi foliation of $\{\psi=s\}$ is dense in $\{\psi=s\}$ for every sufficiently large $s$. 
It then follows from the standard argument based on the maximum principle (see e.g., \cite[Lemma 2.2]{KUeh1}) that there exists no nonconstant holomorphic function on $V\setminus Y$, by shrinking $V$ to a smaller connected neighborhood if necessary. 
Furthermore, if a compact complex hypersurface 
$\Gamma\subset V\setminus Y$ existed,
then $\psi|_\Gamma$ would be pluriharmonic on the compact complex space $\Gamma$,
and thus constant by the maximum principle.
Hence $\Gamma\subset\{\psi=s\}$ for some $s$. Therefore $\Gamma$ would be a union of leaves of the Levi foliation of $\{\psi=s\}$, 
which contradicts the noncompactness of each leaf. This proves assertion $(ii)$.
\end{proof}

\subsection{Demailly--Nemirovski--Fu--Shaw type functions}\label{subsection:conseq_nemi_type_func}

As mentioned in the introduction, 
Nemirovski \cite{N} introduced a function of the form (\ref{eq:Nemirovski_func_original}), and related techniques were later developed by Fu and Shaw \cite{FS} 
in connection with the study of the Diederich--Fornæss exponent. 
Note that such level-set Monge–Ampère expressions were also considered in Demailly's earlier work \cite{Demailly1987} in the context of boundary pluripotential theory. 
In the present paper, we adapt this approach to our setting as follows. 
Let $X$, $Y$, and $\psi$ be as in \S\ref{subsection:standing_seting}, 
and assume that {\bf Property $(\ast)$} holds. 
Given an integer $k\in\{1,2,\dots,n-1\}$ and a smooth real $(1, 1)$-form $\omega>0$ on $X$, as a natural generalization of such functions we define
\[
  F_{k,\omega}(s)
  := \int_{\{\psi = s\}} d^c\psi \wedge (dd^c\psi)^{\wedge k}
     \wedge \omega^{\wedge(n-k-1)},
\]
where the orientation of $\{\psi=s\}$ is taken as that of the boundary of the sublevel set $\{\psi<s\}$; 
that is, the outward normal direction is determined by the form $d\psi$. 
Note that one can easily see the well-definedness of the value of the integral in the right-hand side by virtue of Lemma \ref{lem:psi_levelset_conn}. 

For this function, first let us show the following lemma, which can be shown by the same arguments in \cite{N} and \cite{FS}, see also \cite[(1.9)]{Demailly1987}. 
\begin{lemma}\label{lem:key}
Let $\omega > 0$ and $k$ be as above. 
Assume that $\sqrt{-1}\Theta_{h_D}\geq 0$ and $\omega^{\wedge (n-k-1)}$ is $d$-closed. 
Then the function $F_{k,\omega}(s)$ is nonnegative and nondecreasing for all sufficiently large $s$. 
\end{lemma}

\begin{proof}
The nonnegativity of $F_{k,\omega}$ follows from the above orientation convention and the semi-positivity of $dd^c\psi|_{X\setminus Y}$ and $\omega$ (recall equation $(\ref{eq:Poincare_lelong_ddcpsi})$). 
For the monotonicity, take $1 \ll s_1 < s_2$.  
Using $d(\omega^{\wedge {n-k-1}})=0$ and Stokes' theorem, we have
\begin{align*}
  F_{k,\omega}(s_2) - F_{k,\omega}(s_1)
  &= \int_{\partial\{s_1 < \psi < s_2\}}
        d^c\psi \wedge (dd^c\psi)^{\wedge k} \wedge \omega^{\wedge (n-k-1)} \\
  &= \int_{\{s_1 < \psi < s_2\}}
        d\big(d^c\psi \wedge (dd^c\psi)^{\wedge k} \wedge \omega^{\wedge (n-k-1)}\big) \\
  &= \int_{\{s_1 < \psi < s_2\}}
        (dd^c\psi)^{\wedge (k+1)} \wedge \omega^{\wedge (n-k-1)}, 
\end{align*}
from which the desired monotonicity follows for the same reason. 
\end{proof}

Next let us show the following proposition, which plays a key role in the proofs of the main theorems. 

\begin{proposition}\label{prop:key}
Let $X$, $Y$, $D$, and $\psi$ be as in \S\ref{subsection:standing_seting}, 
and assume that {\bf Property $(\ast)$} holds. 
Let $k\in\{1,2,\dots, n-1\}$ be an integer and $\omega$ be a smooth positive real $(1, 1)$-form on $X$ such that $\omega^{\wedge (n-k-1)}$ is $d$-closed. Then 
\[
\lim_{s\to \infty} F_{k, \omega}(s) = (D^{k+1}.\ \{\omega^{\wedge (n-k-1)}\})
\]
holds. 
\end{proposition}

\begin{proof}
More generally, it suffices to show that 
\[
\lim_{s\to\infty}\int_{\{\psi=s\}}d^c\psi\wedge\alpha
  = \langle T_D, \alpha\rangle
\]
holds for any smooth $d$-closed $(n-1, n-1)$-form $\alpha$ on $X$. 
Indeed, applying this to 
$\alpha=(\textstyle\frac{\sqrt{-1}}{2\pi}\Theta_{h_D})^{\wedge k}\wedge\omega^{\wedge (n-k-1)}$
and using that $dd^c\psi=\textstyle\frac{\sqrt{-1}}{2\pi}\Theta_{h_D}$ on $X\setminus Y$, which follows from equation (\ref{eq:Poincare_lelong_ddcpsi}), yields the desired formula.

Set $W_s:=Y\cup\{\psi>s\}$.  
By Stokes' theorem and the assumption that $\alpha$ is $d$-closed, we have
\[
\int_{\{\psi=s\}}d^c\psi\wedge\alpha
 = -\int_{W_s}d(d^c\psi\wedge\alpha)
 = -\int_{W_s}dd^c\psi\wedge\alpha
\]
(Here recall that we orient each level set $\{\psi=s\}$ as the boundary of the sublevel domain $\{\psi<s\}$ so that $W_s$ lies outside of it). 
Using $(\ref{eq:Poincare_lelong_ddcpsi})$, we obtain
\[
\int_{\{\psi=s\}}d^c\psi\wedge\alpha
 = \langle T_D,\alpha\rangle
   - \int_{W_s}\tfrac{\sqrt{-1}}{2\pi}\Theta_{h_D}\wedge\alpha.
\]
Since $\textstyle\frac{\sqrt{-1}}{2\pi}\Theta_{h_D}\wedge\alpha$ is smooth on $X$,
its coefficients are bounded near $Y$. 
Thus there exists a constant $M>0$ such that
\[
\left|\int_{W_s}\tfrac{\sqrt{-1}}{2\pi}\Theta_{h_D}\wedge\alpha\right|
   \le M\!\int_{W_s}\!dV_X
\]
for any $s \gg 1$, where $dV_X$ is a volume form of $X$.
Since the second term tends to $0$ as $s\to\infty$, the assertion follows.
\end{proof}

By applying Lemma \ref{lem:key} and Proposition \ref{prop:key} in the same manner as in 
\cite{N} and \cite{FS}, we have the following: 
\begin{corollary}\label{cor:key}
Let $X$, $D$, $h_D$, and $\eta$ be as in \S \ref{subsection:standing_seting}. Assume that $\sqrt{-1}\Theta_{h_D}\geq 0$ and {\bf Property $(\ast)$} hold. 
For a Kähler form $\omega$ of $X$ and $k\in\{1,2,\dots, n-1\}$, the following holds: \\
$(i)$ $(D^{k+1}.\ \{\omega\}^{n-k-1})\geq 0$. \\
$(ii)$ If $(D^{k+1}.\ \{\omega\}^{n-k-1}) = 0$, then there exists a positive number $\ve_0$ such that the Levi form of $\{\psi = -\log t\} (=\{\eta = t\})$ has at most $k-1$ positive eigenvalues at any point of $\{\eta = t\}$ for any $t\in (0, \ve_0)$. \\
$(iii)$ If $(D^{k+1}.\ \{\omega\}^{n-k-1}) > 0$, then there exists a positive number $\ve_0$ such that the Levi form of $\{\psi = -\log t\} (=\{\eta = t\})$ has at least $k$ positive eigenvalues at some point of $\{\eta = t\}$ for any $t\in (0, \ve_0)$. 
\end{corollary}

\begin{proof}
By using $\psi$ as in \S \ref{subsection:standing_seting} and $\omega$ together with $k$, 
in what follows we consider the function $F_{k, \omega}$. 
Assertion $(i)$ simply follows from Lemma \ref{lem:key} and Proposition \ref{prop:key}. 
When $(D^{k+1}.\ \{\omega\}^{n-k-1}) = 0$, from Lemma \ref{lem:key} and Proposition \ref{prop:key} it follows that $F_{k, \omega}(s)\equiv 0$ for $s\gg 1$. Therefore, for $s\gg 1$, 
\begin{equation}\label{eq_dcpddcpomegazerointheprof}
\left.d^c\psi\wedge (dd^c\psi)^{\wedge k}\wedge \omega^{\wedge (n-k-1)}\right|_{\{\psi = s\}}=0
\end{equation}
holds at each point $p$ of $\{\psi = s\}$. 
Take local coordinates $(z_1, z_2, \dots, z_n)$ on a neighborhood of $p$ such that 
\[
\omega = \sqrt{-1}\sum_{j=1}^n dz_j\wedge d\overline{z_j}
\]
holds at $p$. By changing the coordinates by using a unitary matrix of order $n$, we may assume that $dz_1$ is parallel to $\del \psi$ at $p$. Furthermore, by applying a unitary matrix of order $n-1$ to $(z_2, z_3, \dots, z_n)$, we may also assume that 
\[
dd^c\psi = \sqrt{-1}\left(a(p)dz_1\wedge d\overline{z_1} + \sum_{j=2}^n b_j(p)dz_1\wedge d\overline{z_j} + \sum_{j=2}^n \overline{b_j(p)}dz_j\wedge d\overline{z_1} + \sum_{j=2}^n \gamma_j(p)dz_j\wedge d\overline{z_j}\right)
\]
holds at $p$, where $a(p)$ and $b_j(p)$'s are constants and $\gamma_j(p)$'s are eigenvalues of the Levi form with respect to $\omega$ at $p$. 
As $d^c\psi$ is parallel to $dz_1 - d\overline{z_1}$ at $p$, we have that 
\[
d^c\psi\wedge (dd^c\psi)^{\wedge k}\wedge \omega^{\wedge (n-k-1)}=c\left(\sum_{(j_1, j_2, \dots, j_k)\in I_k}\prod_{\lambda=1}^k\gamma_{j_\lambda}(p)\right)\,d{\rm Im}\,z_1\wedge \sum_{j=2}^n dz_j\wedge d\overline{z_j} + dz_1\wedge d\overline{z_1}\wedge \alpha, 
\]
where $c$ is a positive constant which depends only on $n$, $k$, and the proportionality factor between $dz_1$ and $\del \psi$, $I_k:=\{(j_1, j_2, \dots, j_k) \mid 2\leq j_1< j_2 < \dots < j_k\leq n\}$, and $\alpha$ is a $(2n-3)$-form. As $dz_1\wedge d\overline{z_1}$ is parallel to $d\psi\wedge d^c\psi$ and $d\psi=0$ holds at $p$ after restricting them to $\{\psi=s\}$, from (\ref{eq_dcpddcpomegazerointheprof}) it follows that 
\[
\prod_{\lambda=1}^k\gamma_{j_\lambda}(p) = 0
\]
holds at any point $p$ of $\{\psi = s\}$ for any $(j_1, j_2, \dots, j_k) \in I_k$ and $s\gg 1$ (here we used the fact that each $\gamma_j(p)$ is non-negative, since $\psi$ is plurisubharmonic on a neighborhood of $p$). 
Therefore assertion $(ii)$ follows. 

When $(D^{k+1}.\ \{\omega\}^{n-k-1}) > 0$, again from Lemma \ref{lem:key} and Proposition \ref{prop:key} it follows that $F_{k, \omega}(s) > 0$ for $s\gg 1$. Therefore, for $s\gg 1$, 
\[
\left.d^c\psi\wedge (dd^c\psi)^{\wedge k}\wedge \omega^{\wedge (n-k-1)}\right|_{\{\psi = s\}} > 0
\]
holds on some point of $\{\psi = s\}$. 
Therefore, by the same argument as above, we have that 
\[
\prod_{\lambda=1}^k\gamma_{j_\lambda}(p) > 0
\]
holds at a point $p$ of $\{\psi = s\}$ for some $(j_1, j_2, \dots, j_k) \in I_k$, for any $s\gg 1$, from which assertion $(iii)$ follows. 
\end{proof}


\section{Proof of main theorems}\label{section:prf_mainresults}

In this section, we give proofs of the main theorems stated in \S 1. 
For every theorem, the assertion is trivial when the dimension $n$ of the complex manifold $X$ is one. 
Therefore, we shall restrict ourselves to the case $n>1$ in what follows.

\subsection{A general result and proof of Theorem \ref{thm:main_vgeneral}}

Here we first show the following general result, in which $X$ is not assumed to be compact. 

\begin{theorem}\label{thm:general}
Let $X$, $Y$, $D$ be as in \S\ref{subsection:standing_seting}. 
Assuming that $[D]$ is semi-positive, take a $C^\infty$ Hermitian metric $h_D$ of $[D]$ with $\sqrt{-1}\Theta_{h_D} \geq 0$ and the function $\psi$ as in \S \ref{subsection:standing_seting}. 
Assume also that $X$ is Kähler, {\bf Property $(\ast)$} holds, and that $D$ is numerically trivial along $Y$ with respect to Kähler classes of $X$. 
Then the level set $\{\psi = s\}$ is a connected Levi-flat hypersurface for $s \gg 1$, and one of the following holds: \\
$(a)$ For any positive number $M>0$ there exists a real number $s$ with $s > M$ such that the level set $\{\psi = s\}$ admits a non-constant real-valued $C^\infty$ function that is leafwise constant with respect to the Levi foliation.  \\
$(b)$ There exists a neighborhood $V$ of $Y$ such that $[D]|_V$ is unitary flat. 
\end{theorem}

\begin{proof}
By Lemma \ref{lem:psi_levelset_conn}, the level set $\{\psi = s\}$ is a connected compact real hypersurface of $X\setminus Y$ for all $s \gg 1$.
As $D$ is numerically trivial along $Y$ with respect to Kähler classes of $X$, 
\[
(D^2, \{\omega\}^{n-2}) = \sum_{\nu=1}^N m_\nu\cdot \int_{Y_\nu} c_1(D) \wedge \{\omega\}^{\wedge (n-2)} = 0
\]
holds for a Kähler form $\omega$ of $X$, where $m_\nu$ is the coefficient of $Y_\nu$ in the irreducible decomposition of $D$. 
Therefore, by applying Corollary \ref{cor:key} $(ii)$ with $k=1$, it follows that $\{\psi = s\}$ is a connected Levi-flat hypersurface for $s \gg 1$. 

In what follows we show assertion $(b)$ by assuming that $(a)$ does not hold. 
Set $V := Y \cup \{\psi > s_0\} = \{\eta < e^{-s_0}\}$ for sufficiently large $s_0>0$. 
As any $s \gg 1$ is a regular value of $\psi$ by Lemma \ref{lem:psi_levelset_conn} and the level sets are Levi-flat, it follows from a standard pointwise linear-algebraic argument that there exists a non-negative real-valued function $f$ on $V\setminus Y$ that is leafwise constant with respect to the Levi foliations for all level sets such that 
\[
\sqrt{-1}\ddbar \psi = f\cdot \sqrt{-1} \del \psi \wedge \delbar \psi
\]
holds on $V\setminus Y$ (for the details, see the proof of \cite[Lemma 4.2]{Ko2022}, or Lemma 3.1 and the arguments around Lemma 3.2 in \cite{Ko2024} for example). 
Note that locally this function $f$ can be expressed as follows on the complement $W\setminus Y$ for a small neighborhood $W$ of a regular point of $Y$ in $V$: 
\begin{equation}\label{eq:local_express_f}
f = \frac{\sqrt{-1}\ddbar \psi}{\sqrt{-1} \del \psi \wedge \delbar \psi} 
= \frac{\sqrt{-1}\ddbar \vp}{\sqrt{-1} (-mdw/w + \del \vp) \wedge (-md\overline{w}/\overline{w} + \delbar \vp)} 
= \eta^{1/m}\cdot A, 
\end{equation}
where $w$ is a local defining function of $Y$ on $W$, $m$ is the multiplicity of $D|_W$ along $Y$, $\vp$ is the local weight function of $h_D$ such that $\psi = -m\log |w|^2 + \vp$ holds on $W\setminus Y$, and $A$ is a smooth function on $W$. 

As $f$ is leafwise constant, from our assumption that condition $(a)$ does not hold, it follows that there exists a function $G\colon (s_0, \infty) \to \mathbb{R}$ such that $f=G\circ \psi$ holds on $V\setminus Y$, by replacing $s_0$ with a larger number if necessary. 
Note that, as it is clear from the local expression (\ref{eq:local_express_f}) of $f$, it follows that 
\[
G(s) = \exp(-s/m)\cdot B(s) = t^{1/m}\cdot \widehat{B}(t^{1/2m}), 
\]
where the parameters $s$ and $t$ are related by (\ref{eq:relation_sandt}), 
holds for some functions $B$ and $\widehat{B}$. 
Note also that $\widehat{B}(\tau)$ is defined and continuous also at $\tau=0$, since the function $A$ is smooth on $W$. 

In this case, by arguments similar to those in the proof of \cite[Lemma 4.2]{Ko2022}, one can solve a suitable ordinary differential equation to construct a function $\chi$ such that $\chi\circ\psi$ is pluriharmonic on $V\setminus Y$ (see also the arguments in \cite[\S 5]{MongodiSlodkowskiTomassini18} and \cite[\S 5.1]{MongodiTomassini20}). 
However, in the present situation we need more precise information on the behavior of the modified function $\chi\circ\psi$ near $Y$.
For this reason, we shall essentially repeat the same argument, but now focusing on the function $\eta$ instead of $\psi$, and look for a function $\widehat{\chi}$ such that $\widehat{\chi}\circ\eta$ is pluriharmonic on $V\setminus Y$, while paying particular attention to the behavior of $\widehat{\chi}(t)$ around $t=0$. 

Let $g$ be the function on $V\setminus Y$ such that 
\[
\sqrt{-1}\ddbar \eta = g\cdot \sqrt{-1}\del \eta \wedge \delbar \eta. 
\]
As $\del \eta =  -\eta\del \psi$, $\delbar \eta =  -\eta\delbar \psi$, and 
$\ddbar \eta = \eta(-\ddbar \psi + \del\psi\wedge \delbar \psi)$, 
we have 
\[
g = \frac{\sqrt{-1}\ddbar \eta}{\sqrt{-1}\del \eta \wedge \delbar \eta}
=\frac{\eta\cdot (-f+1)}{(-\eta)^2}=\frac{1-f}{\eta}. 
\]
Therefore, by letting
\[
\widehat{G}(t) := \frac{1-G(-\log t)}{t} = \frac{1}{t} - t^{-1+(1/m)}\cdot \widehat{B}(t^{1/2m}), 
\]
we have $g = \widehat{G} \circ \eta$ on $V\setminus Y$. 

As $\widehat{B}(\tau)$ is defined and continuous also at $\tau=0$, the function 
\[
\rho(t) 
:=\frac{1}{t}\cdot \exp\int_0^t \tau^{-1+(1/m)}\cdot \widehat{B}(\tau^{1/2m})\,d\tau
\]
is well-defined; More precisely, we have 
\[
\rho(t) = \frac{1}{t}\cdot \exp\left(t^{1/m}\cdot \beta(t)\right)
\]
holds for a continuous function $\beta$ with $\beta(0) = m\widehat{B}(0)$. 
As $(\log \rho)'=-\widehat{G}$ by construction, we have 
$\rho' = -\widehat{G}\cdot \rho$. 
Therefore by letting 
\[
\widehat{\chi}(t) := \log t + \int_0^t\left(\rho(\tau) - \frac{1}{\tau}\right)\,d\tau, 
\]
where the well-definedness of the second term of the right-hand side can be easily checked by a fundamental argument, we have that the function 
\[
\widehat{\psi} := -\widehat{\chi}\circ \eta
\]
satisfies that $\widehat{\psi}-\psi$ is a bounded function on a neighborhood of $Y$ and that 
\begin{align*}
\ddbar \widehat{\psi} &= -(\widehat{\chi}'\circ \eta)\cdot \ddbar \eta -(\widehat{\chi}''\circ \eta)\cdot \del \eta\wedge \delbar \eta 
=-(\rho\circ \eta)\cdot \ddbar \eta -(\rho'\circ \eta)\cdot \del \eta\wedge \delbar \eta \\
&=(\rho\circ \eta)\cdot \left(-\ddbar \eta +(\widehat{G}\circ \eta)\cdot \del \eta\wedge \delbar \eta\right)
=(\rho\circ \eta)\cdot \left(-\ddbar \eta +g\cdot \del \eta\wedge \delbar \eta\right) \equiv 0
\end{align*}
holds on $V\setminus Y$. 
On a neighborhood $U$ of each point of $Y$, from this it follows that $\widehat{\psi}+\log|\sigma|^2$ is bounded on $U$ and pluriharmonic on $U\setminus Y$, where $\sigma$ is a local defining function of $D$ on $U$. 
As we have $\widehat{\psi}+\log|\sigma|^2$ is pluriharmonic also on $U$ in this case (see \cite[p. 45 (5.24) Theorem]{Demaillyagbook} for example), the Hermitian metric on $[D]|_V$ that corresponds to $\widehat{\psi}$ in the correspondence in Remark \ref{rmk:psi-h-corresp} is a flat metric. 
\end{proof}

By applying Theorem \ref{thm:general}, one can show Theorem \ref{thm:main_vgeneral} as follows. 

\begin{proof}[Proof of Theorem \ref{thm:main_vgeneral}]
Let $X$, $Y$, and $D$ be as in Theorem \ref{thm:main_vgeneral}. 
Take a $C^\infty$ Hermitian metric $h_D$ of $[D]$ with $\sqrt{-1}\Theta_{h_D} \geq 0$ and the function $\psi$ as in \S \ref{subsection:standing_seting}. 
Without loss of generality, we may assume that $\psi$ satisfies {\bf Property $(\ast)$} (see Remark \ref{rmk:wlogwmaV0conn}). 
Then it is enough to show assertion $(i)$ of Theorem \ref{thm:main_vgeneral} when condition $(a)$ of 
Theorem \ref{thm:general} holds (Note that it follows from equation (\ref{eq:keisan_D2alphan-2}) that the assumption of the theorem implies that $D$ is numerically trivial along $Y$ with respect to Kähler classes of $X$; also refer to the discussion right after Definition \ref{def:num_triv_along_Y_wrt_K}). 
For this purpose, for a given neighborhood $V$ of $Y$, take a sufficiently large $s \gg 1$ such that the level set $\{\psi = s\}$ is contained in $V \setminus Y$. 
Note that we may assume that this level set is a connected compact real hypersurface by virtue of Lemma \ref{lem:psi_levelset_conn}. 
By condition $(a)$, we may also assume that there exists a non-constant $C^\infty$ function $h\colon \{\psi = s\} \to \mathbb{R}$ that is leafwise constant with respect to the Levi foliation. 
As $h$ is a continuous function from a compact connected space, 
the image ${\rm Image}\,h := h(\{\psi = s\})$ is a closed interval in $\mathbb{R}$ and 
the set of the critical values ${\rm Crit}\,h$ is a closed subset of ${\rm Image}\,h$. 
Moreover, from Sard's theorem, it follows that the complement ${\rm Image}\,h\setminus {\rm Crit}\,h$ contains a non-empty open interval $I$. 
As the map $h|_{h^{-1}(I)}\colon h^{-1}(I) \to I$ is a proper submersion,
Ehresmann's fibration theorem implies that it defines a smooth locally trivial fibration,
and the desired assertion follows by its fibers (Here we note that $\Gamma_\tau := h^{-1}(\tau)$ is the union of some leaves of Levi foliations for each $\tau \in I$, since $h$ is leafwise constant, and thus this fibration defines a smooth one-parameter family of non-singular compact complex hypersurfaces). 
\end{proof}

At the end of this subsection, let us show the following for the geometric structure of $X$ in the case of $(a)$ in Theorem \ref{thm:general} when $X$ is compact Kähler, as a preparation to the proofs of the other main theorems. 
\begin{lemma}\label{lem:JDGlem36}
Let $X$ be a connected compact Kähler manifold, and $D$ be an effective divisor of $X$ with connected support $Y$ such that $D$ is numerically trivial along $Y$ with respect to Kähler classes of $X$. 
Assuming that $[D]$ is semi-positive, take a $C^\infty$ Hermitian metric $h_D$ on $[D]$ with $\sqrt{-1}\Theta_{h_D}\geq 0$. For the function $\psi$ as in \S \ref{subsection:standing_seting}, assume that {\bf Property $(\ast)$} holds and condition $(a)$ of Theorem \ref{thm:general} holds. Then there exists a fibration $\Phi\colon X \to R$ onto a compact Riemann surface $R$ such that $Y$ is a fiber of $\Phi$. 
\end{lemma}

\begin{proof}
By \cite[Lemma 3.6]{Ko2024}, there exists a leaf $Y'$ of the Levi foliation of the level set $\{\psi = s\}$ for some $s \gg 1$ and a fibration $\Phi\colon X \to R$ onto a compact Riemann surface $R$ such that $Y'$ is a fiber of $\Phi$. 
As $Y$ is connected, it follows from Remmert's proper mapping theorem (see \cite[Chapter II (8.8)]{Demaillyagbook}) that the image $\Phi(Y)$ of $Y$ is either $R$ or a point of $R$. 
As $Y\cap Y'=\emptyset$, we have $\Phi(Y)\not=R$, from which the assertion follows. 
\end{proof}

\begin{remark}
Here we note that the fibration $\Phi$ in the proof of Lemma \ref{lem:JDGlem36} can also be obtained by directly applying Neeman's argument \cite[p. 109--110]{Nee} (see also \cite[Theorem 2.1]{Tot}), which is used in the proof of \cite[Lemma 3.6]{Ko2024}, to $\Gamma_\tau$'s as in the proof of Theorem \ref{thm:main_vgeneral}. 
\end{remark}

\subsection{Proof of Theorem \ref{thm:main_gen}}
Let $X$, $Y$, and $D$ be as in Theorem \ref{thm:main_gen}. 
Since all the conditions $(i)$, $(ii)$, and $(iii)$ concern a neighborhood of $Y$ (recall Lemma \ref{lem:sp_local} and Remark \ref{rmk:wlogwmaYconn} for condition $(i)$), to prove the theorem we may assume that both $X$ and $Y$ are connected. 
Note that $D$ is numerically trivial along $Y$ with respect to Kähler classes of $X$, by virtue of Lemma \ref{lem:num1_flatrestr}. 

First let us show the implication $(iii)\implies (ii)$. 
Take a system $\{(V_j, s_j)\}$ as in $(iii)$. 
Denote by $D'$ the divisor on $X$ that is locally defined by $s_j$ on each $V_j$. 
Then we have ${\rm Supp}\,D'=Y$ and $[D']|_V$ is unitary flat for $V:=\textstyle\bigcup_jV_j$ by construction. 
Take a flat metric $h_{D'}$ on $[D']|_V$. As $\textstyle\frac{\sqrt{-1}}{2\pi}\Theta_{h_{D'}}\equiv 0$ represents the class $c_1^{\mathbb{R}}([D']|_V)$, we have $c_1^{\mathbb{R}}([D']|_Y)=0$. 
In particular, $D'$ is numerically trivial along $Y$ with respect to Kähler classes of $X$. 
Therefore, from Lemma \ref{lem:from_linalg} $(ii)$, it follows that $D=qD'$ holds for some positive rational number $q$. As $h_{D'}^q$ defines a flat metric on $[D]|_V$, $(ii)$ follows from Lemma \ref{lem:flatlb_funds} $(iii)$. 

Next let us show the implication $(ii)\implies (i)$. 
If $[D]|_V$ is unitary flat for a neighborhood $V$ of $Y$, it is clear by definition that $[D]|_V$ is semi-positive. Therefore $[D]$ is also semi-positive by Lemma \ref{lem:sp_local}. 

Finally we show the implication $(i)\implies (iii)$. 
Take a $C^\infty$ Hermitian metric $h_D$ of $[D]$ with $\sqrt{-1}\Theta_{h_D} \geq 0$ and the function $\psi$ as in \S \ref{subsection:standing_seting}. 
Without loss of generality, we may assume that $\psi$ satisfies {\bf Property $(\ast)$} (see Remark \ref{rmk:wlogwmaV0conn}). 
Then either $(a)$ or $(b)$ of Theorem \ref{thm:general} holds. 
In the case where $(b)$ holds, one can easily show that $(iii)$ holds. 
When $(a)$ holds, by Lemma \ref{lem:JDGlem36} there exists a fibration $\Phi\colon X \to R$ onto a compact Riemann surface $R$ such that $Y = \Phi^{-1}(p)$ holds for some $p\in R$. 
Assertion $(iii)$ follows in this case by considering (the local expressions of) the function $s := w\circ \Phi$, where $w$ is a local coordinate function on a neighborhood of $p$ with $w(p)=0$. 
\qed

\subsection{Proof of Theorem \ref{thm:main_tor}}

For proving Theorem \ref{thm:main_tor}, let us first show the following proposition, which concerns the case where the support $Y$ of $D$ is connected. 
\begin{proposition}\label{prop:_maeno_main_tor}
Let $X$ be a connected compact Kähler manifold and $D$ be an effective divisor on $X$ with connected support $Y\subset X$ such that $[mD]|_Y$ is holomorphically trivial for some positive integer $m$. 
Then the following conditions are equivalent: \\
$(i)$ The line bundle $[D]$ is semi-positive. \\
$(ii)$ There exists a fibration $\pi\colon X\to R$ onto a compact Riemann surface $R$ such that $\pi^*\{p\} = aD$ as divisors on $X$ for some point $p\in R$ and some positive rational number $a$.
\end{proposition}

\begin{proof}
First let us show the implication $(ii)\implies (i)$. 
Take a fibration $\pi\colon X\to R$, a point $p\in R$, and a positive rational number $a$ as in condition $(ii)$. 
As the line bundle $[\{p\}]$ is of degree one, we have that it is ample. 
Therefore, by considering the pull-back of the Fubini--Study metric by the embedding map into a projective space, one can construct a $C^\infty$ Hermitian metric $h_p$ such that $\sqrt{-1}\Theta_{h_p}>0$ (see \cite[\S 3]{Demailly} for the details). As the metric $h_D := (\pi^*h_p)^{1/a}$ on $[D]$ satisfies $\sqrt{-1}\Theta_{h_D}=\textstyle\frac{1}{a}\pi^*\sqrt{-1}\Theta_{h_p} \geq 0$, $[D]$ is semi-positive. 

Next let us show the converse. 
Assume that $[D]$ is semi-positive. 
As $[mD]|_Y$ is topologically trivial for some $m>0$, it follows that $c_1^{\mathbb{R}}([D]|_Y) = \textstyle\frac{1}{m}c_1^{\mathbb{R}}([mD]|_Y)=0$. 
Therefore we have that $D$ is numerically trivial along $Y$ with respect to Kähler classes of $X$, and that ${\rm nd}(D) = 1$, by virtue of Lemma \ref{lem:num1_flatrestr}. 
Thus, by Theorem \ref{thm:main_gen}, 
$[D]|_V$ is unitary flat for some neighborhood $V$ of $Y$. 
In what follows we use the notation in \S \ref{subsection:semilocal_nbhd_Y_Dflat} such as $\mathcal{F}$ and $\rho_{D|_Y}$. 

From the assumption and Lemma \ref{lem:flatlb_funds} $(ii)$ (and Remark \ref{rmk:inj_flatlb_hollb}), we have that $(\rho_{D|_Y})^m = \rho_{mD|_Y}$ is the trivial representation for a positive integer $m$. Therefore the image of $\rho_{D|_Y}$ is a finite subgroup of ${\rm U}(1)$. 
By Lemma \ref{lem:scv_monodromy_onV} $(i)$, by shrinking $V$ if necessary, there exists a proper surjective holomorphic map $p\colon V \to \Delta$ onto a disc $\Delta \subset \mathbb{C}$ centered at the origin such that $Y = p^{-1}(0)$. 
From the same argument as in the proof of Lemma \ref{lem:sp_local}, 
one can construct a $C^\infty$ plurisubharmonic function $\psi\colon X\setminus Y\to \mathbb{R}$ with {\bf Property $(\ast)$} such that $\psi|_{W\setminus Y} = -\log |p|^2$ holds for a neighborhood $W$ of $Y$ in $V$. 
In what follows, as an Hermitian metric on $[D]$, we use $h_D$ that is constructed from $\psi$ by the argument in Remark \ref{rmk:psi-h-corresp}. 
Then, by considering the restriction of the function ${\rm Re}\,p$ for example, we have that condition $(a)$ of Theorem \ref{thm:general} holds. 
Therefore it follows from Lemma \ref{lem:JDGlem36} that there exists a fibration $\pi\colon X \to R$ onto a compact Riemann surface $R$ such that $Y=\pi^{-1}(p)$ for some $p\in R$. 
Assertion $(ii)$ follows by applying Lemma \ref{lem:from_linalg} $(ii)$ to the divisors $D$ and $D':=\pi^*\{p\}$. 
\end{proof}

From Proposition \ref{prop:_maeno_main_tor}, one can deduce Theorem \ref{thm:main_tor} as follows. 

\begin{proof}[Proof of Theorem \ref{thm:main_tor}]
Let $X$, $Y$, $D$ be as in Theorem \ref{thm:main_tor}. 
First let us show the implication $(ii)\implies (i)$. 
Take a fibration $\pi\colon X\to R$ and a divisor $D_R$ as in $(ii)$. 
As $D\not=0$ and $D_R$ is an effective divisor on a compact Riemann surface, it follows that the degree of $D_R$ is positive. 
Therefore the semi-positivity of $[D]$ follows by the same argument as in the proof of Proposition \ref{prop:_maeno_main_tor} $(ii)\implies (i)$. 

Next let us show assertion $(ii)$ by assuming that $[D]$ is semi-positive. 
Let $D = D_1 + D_2 + \cdots + D_M$ be the decomposition as in Remark \ref{rmk:wlogwmaYconn}. Then, by the fact mentioned in this remark, it follows that $[D_1]$ is semi-positive. 
Therefore it follows from Proposition \ref{prop:_maeno_main_tor} that there exists a fibration $\pi\colon X\to R$ onto a compact Riemann surface $R$ such that $D_1 = a_1\pi^*\{p_1\}$ as divisors on $X$ for some point $p_1\in R$ and some positive rational number $a_1$. 
By the argument as in the proof of Lemma \ref{lem:JDGlem36}, it follows from Remmert's proper mapping theorem that $\pi({\rm Supp}\,D_\mu) = \{p_\mu\}$ holds for some points $p_\mu \in R$ for each $\mu \in\{2, 3, \dots, M\}$. 
Applying the argument as in the proof of Proposition \ref{prop:_maeno_main_tor}, one can deduce from Lemma \ref{lem:from_linalg} $(ii)$ that $D_\mu = a_\mu\pi^*\{p_\mu\}$ holds for some positive rational numbers $a_\mu$ for each $\mu$. 
Taking a positive integer $m'$ such that $m' a_\mu\in \mathbb{Z}$ and letting 
$D_R := m' (a_1 \{p_1\} + a_2 \{p_2\} + \cdots +a_M\{p_M\})$, we have the assertion. 
\end{proof}

\subsection{Proof of Theorem \ref{cor:main}}
When $D$ is not numerically trivial along $Y$ with respect to Kähler classes of $X$, 
the Hartogs type extension theorem holds on $X\setminus Y$ by virtue of \cite[Theorem 0.2]{O20122} (also refer to \cite[Theorem 3.1]{F2025}). 
Note that this fact can also be shown by applying our argument concerning Demailly--Nemirovski--Fu--Shaw type functions and a more fundamental Hartogs type result due to Ohsawa \cite{O2012} (see also \cite[\S 5.2.2]{Obook}) as follows: 
As $[D]$ is semi-positive and $D$ is not numerically trivial along $Y$ with respect to Kähler classes of $X$, we have 
\begin{equation}\label{eq:prothm14positive}
(D^2, \{\omega\}^{n-2})=\sum_\nu m_\nu \int_{Y_\nu} c_1(D)\wedge \{\omega\}^{n-2} > 0, 
\end{equation}
where $m_\nu$ is the coefficient of $Y_\nu$ in the irreducible decomposition of $D$. 
Let $\psi$ be the function as in \S\ref{subsection:standing_seting} which which is constructed from an Hermitian metric with semi-positive Chern curvature. 
Without loss of generality, we may assume that $\psi$ satisfies {\bf Property $(\ast)$} (see Remark \ref{rmk:wlogwmaV0conn}). 
Then, by applying Corollary \ref{cor:key} $(iii)$ with $k=1$, from (\ref{eq:prothm14positive}) we have that the Levi form of $\{\psi = s\}$ has at least $1$ positive eigenvalue at some point of $\{\psi = s\}$ for $s\gg 1$. Thus the Hartogs type extension theorem on $X\setminus Y$ follows from \cite[Theorem 0.1]{O2012}. 

Next let us consider the case where $D$ is numerically trivial along $Y$  with respect to Kähler classes of $X$. 
Then it follows from Lemma \ref{lem:num1_flatrestr} and 
Theorem \ref{thm:main_gen} that $[D]|_V$ is unitary flat for some neighborhood $V$ of $Y$. 
Let $\rho_{D|_Y}$ be the representation as in \S \ref{subsection:semilocal_nbhd_Y_Dflat}. 
When its image is a finite subgroup of ${\rm U}(1)$, from Proposition \ref{prop:_maeno_main_tor} it follows that assertion $(i)$ holds, since $[mD]|_Y$ is holomorphically trivial for some positive integer $m$ in this case. 
When the image of $\rho_{D|_Y}$ is infinite, the Hartogs type extension theorem on $X\setminus Y$ clearly holds, because $\mathcal{O}_X(V\setminus Y)=\mathbb{C}$ holds for a sufficiently small open connected neighborhood $V$ of $Y$ by Lemma \ref{lem:scv_monodromy_onV} $(ii)$ in this case. 

Finally, let us show that conditions $(i)$ and $(ii)$ cannot be satisfied simultaneously. 
Assume that there exists a fibration $\pi\colon X \to R$ and a point $p\in R$ as in $(i)$. 
On a coordinate open disc $\Delta \subset R$ centered at $0$, take a holomorphic function $f \in \mathcal{O}_R(\Delta)$ that cannot be analytically continued across any point of the boundary of $\Delta$. 
Then $f\circ \pi$ cannot be an element of the image of the restriction map $\mathcal{O}_X(X\setminus Y) \to \mathcal{O}_X(\pi^{-1}(\Delta)\setminus Y)$, from which the assertion follows. 
\qed


\section{Examples and discussions}

\subsection{Examples}
In this subsection, we discuss several examples concerning the semi-positivity of the line bundle $[D]$ associated with an effective divisor $D$ on a projective complex manifold $X$.

Note first that if $[D]$ is semi-positive, then it is nef, but the converse does not hold in general. 
The first example of a nef but non--semi-positive line bundle seems to have been discovered by Yau \cite[p. 228]{Yau} (also refer to the explanation around \cite[Theorem 2.1]{T}). 
The same counterexample was later rediscovered by Demailly, Peternell, and Schneider \cite[Example 1.7]{DPS94}, who in addition proved that in this example $c_1(L)$ contains only one closed positive current. 
This phenomenon was subsequently interpreted within the framework of Ueda theory and generalized in \cite{Ko2015} (see also \cite[\S7.1]{Ko2024}). 
Note that their counterexample $(X, Y, D)$ satisfies that $D=Y$, $[Y]|_Y$ is holomorphically trivial, and that $X\setminus Y$ is a Stein manifold. Therefore, by applying Theorem \ref{thm:main_tor}, 
we can give another proof of the non-semi-positivity of $[D]$ in this example as follows: 
Condition $(ii)$ of Theorem \ref{thm:main_tor} does not hold for this example, since otherwise one can construct a compact complex curve in $X\setminus Y$ by considering a fiber of the fibration, which contradicts the Steinness of $X\setminus Y$. 
Therefore from this theorem it follows that $[D]$ is not semi-positive. 

Even when $[D]$ is semi-positive, it does not necessarily admit a real-analytic Hermitian metric $h_D$ with $\sqrt{-1}\Theta_{h_D} \ge 0$.
Indeed, a counterexample was given by Brunella \cite{B} by blowing up $\mathbb{P}^2$ at nine points in very general position and by considering its anti-canonical bundle. 
Since such rational surface examples form the fundamental motivation for the present study, we recall their setting in some detail. 
Let $C \subset \mathbb{P}^2$ be a smooth cubic curve, and $p_1, p_2, \dots, p_9 \in C$ be nine regular points of $C$.
Let $X$ be the blow-up of $\mathbb{P}^2$ at these nine points, and denote by $Y$ the strict transform of $C$. 
We set $D := Y$. 
Note that one has $K_X \cong [-D]$. 

In this example, when $Y$ is non-singular, the divisor $D$ is nef. 
Brunella first obtained a characterization of the semi-positivity of $[D]$ in this situation in \cite{B}, whose generalization is given in \cite[Corollary 1.5]{Ko2024}, which is further generalized as Theorems \ref{thm:main_gen} and \ref{thm:main_vgeneral} in the present paper. 
In this example with non-singular $Y$, a sufficient condition for the validity of $(iii)$ in Theorem \ref{thm:main_gen} can be described in terms of the dynamical and irrational-theoretical behavior of the sequence $\{[mY]|_Y\}_{m\in\mathbb{Z}}$ in the Picard group of the elliptic curve $Y$ \cite{A, U} (see also \cite{Ko2020, GS1, GS2, Og} for the generalizations). 
In \cite{Ko2023}, this sufficient condition was further relaxed, though it still remains open whether $[D]$ is necessarily semi-positive in this example, see also \cite{De2015}. 
Note that the higher-dimensional analogue of the above phenomenon can be observed in the blow-up of a del Pezzo manifold of degree $1$ at one point, as shown in \cite{Ko2020}. 

When the cubic curve $C$ is singular in this example, the nefness and semi-positivity of $[D]$ have been completely characterized in \cite{Ko2017} and \cite{Ko2023} (Note that, in this singular case, examples of nef but non–semi-positive $[D]$ are also known to exist). 

\subsection{Torsion-type assumption in Theorem \ref{thm:main_tor} and connection with the abundance conjecture}\label{section:obs_abundance}

Let $X$ be an $n$-dimensional compact Kähler manifold. 
A holomorphic line bundle $L$ on $X$ is said to be \emph{$\mathbb{Q}$-effective} 
if there exists a positive integer $\ell$ such that 
$\Gamma(X, \mathcal{O}_X(L^{\ell})) \neq 0$. 
With this terminology, 
Theorems \ref{thm:main_gen} and \ref{thm:main_tor} can be regarded as results concerning 
$\mathbb{Q}$-effective nef line bundles $L$ on compact Kähler manifolds $X$ 
with ${\rm nd}(L)=1$. 
In particular, Theorem \ref{thm:main_tor} may be regarded as a result giving a sufficient condition, formulated in terms of a torsion-type property of the associated divisor, for the equivalence between the semi-positivity and semi-ampleness of such a line bundle on $X$. 

In this section, we discuss how the sufficient condition mentioned above arises naturally from the viewpoint of the abundance conjecture for the case where $L=K_X$ is $\mathbb{Q}$-effective and ${\rm nd}(K_X)=1$, under some technical assumptions for simplicity. 

\begin{remark}\label{rmk:qeffopen}
It is worth recalling the relation between the \emph{non-vanishing conjecture} and the abundance conjecture when ${\rm nd}(K_X)=1$.
Roughly speaking, the present study of the abundance conjecture proceeds in two steps: proving non-vanishing and then establishing semi-ampleness under that assumption. 
For non-vanishing, the case of numerical dimension one was studied in detail in \cite[\S 6]{LP}, and the results there have been generalized to Kähler varieties in \cite{HLL}. For the other part, 
we refer to \cite[Theorem 1.4]{Bi}, \cite{GM}, \cite{HLL}, and \cite{LX}. 
While various effective approaches and results have been obtained in this direction as above, 
it seems that most of the known results have been obtained by induction on the dimension. 
The inductive arguments concern the validity of the conjectures for 
\emph{all} varieties in a given dimension, rather than for individual cases. 
Therefore, even if a given manifold $X$ satisfies ${\rm nd}(K_X)=1$ and $K_X$ is $\mathbb{Q}$-effective, 
it still seems to remain unclear whether one can conclude the semi-ampleness of $K_X$, except for the cases settled in low dimensions \cite{LX}, particularly in higher dimensions and in the Kähler setting. 
\end{remark}

Let $X$ be a connected projective complex manifold such that $K_X$ is nef, $\mathbb{Q}$-effective, and ${\rm nd}(K_X)=1$. 
Take a positive integer $\ell$ and an effective divisor $D$ such that $K_X^{\ell}=[D]$. 

First let us consider the case where the conclusion of the abundance conjecture holds: i.e., when $K_X$ is semi-ample. 
In this case, as ${\rm nd}(K_X)=1$, one obtains a fibration $\pi \colon X \to R$ onto a compact Riemann surface $R$ and an ample divisor $D_R$ of $R$ such that $\pi^*D_R = m'D$ holds for some positive integer $m'$, as in Theorem \ref{thm:main_tor} (ii), since the Kodaira dimension is one (see \cite[Chapter V \S 2.a]{Nakayama04}). 
Then, as $Y:={\rm Supp}\,D$ is the union of some fibers of $\pi$, we have that $[m'D]|_Y$ is holomorphically trivial. 

Next, let us consider the situation without assuming that $K_X$ is semi-ample. 
In what follows, for simplicity, assume that the support $Y$ of $D$ is connected and non-singular, so that the \emph{adjunction formula} $(K_X\otimes [Y])|_Y \cong K_Y$ holds. 
Note that $D=bY$ holds for a positive integer $b$, and that $c_1^{\mathbb{R}}([Y]|_Y)=0$ holds in this case (see Remark \ref{rmk_torsion_normalbundle}). As
\[
K_Y^\ell \cong (K_X^\ell\otimes [\ell Y])|_Y = [D+\ell Y]|_Y=[(b+\ell)Y]|_Y, 
\]
we have 
\[
c_1^{\mathbb{R}}(K_Y) = \frac{b+\ell}{\ell}c_1^{\mathbb{R}}([Y]|_Y) = 0, 
\]
namely ${\rm nd}(K_Y) = 0$. In this case, it is known that $K_Y$ is $\mathbb{Q}$-effective (the solution of the abundance conjecture in the case of numerical dimension zero \cite[Chapter V]{Nakayama04}). 
Therefore one can take a positive integer $\lambda$ such that $K_Y^\lambda$ is holomorphically trivial. 
Thus, by letting $m:=\lambda (b+\ell)$, we have that 
\[
[mD]|_Y = [\lambda b(b+\ell)Y]|_Y \cong K_Y^{\lambda b\ell}
\]
is also holomorphically trivial. 

Although the above discussion was restricted to the case where $Y$ is connected and non-singular, 
we believe that it suggests that the torsion-type assumption in Theorem \ref{thm:main_tor} arises quite naturally from the viewpoint of the abundance conjecture. 

\subsection{Application}

In view of the study of the abundance conjecture as mentioned in Remark \ref{rmk:qeffopen}, 
it seems to be natural to expect that one can obtain applications of our main results by combining them with non-vanishing type results. In this subsection, let us consider such an application by combining them with the following theorem due to Höring, Lazić, and Lehn: 
\begin{theorem}[= a part of {\cite[Theorem C]{HLL}}, when $X$ is non-singular]\label{thm:HLL}
Let $X$ be a compact Kähler manifold with $c_1(X)=0$, 
and $L$ be a nef line bundle on $X$ with ${\rm nd}(L)=1$. 
If $\chi(X, \mathcal{O}_X)\not=0$, then there exists an effective divisor $D$ on $X$ such that $c_1^{\mathbb{R}}(D) = c_1^{\mathbb{R}}(L^m)$ holds for some positive integer $m$. 
\end{theorem}
Note that this theorem is stated more generally for $\mathbb{Q}$-factorial compact Kähler klt varieties in \cite{HLL}. 
Note also that the assumption that $X$ is not uniruled, which is originally assumed in \cite[Theorem C]{HLL}, is dropped here by virtue of \cite[Theorem 1.1]{Ou}, as we restrict ourselves to the case of non-singular $X$ here. For example, an even-dimensional Calabi--Yau manifold $X$ in the strict sense satisfies the conditions in this theorem, since $K_X$ is holomorphically trivial and $\chi(X, \mathcal{O}_X)=2$ holds for such $X$. 

By applying Theorem \ref{thm:main_gen} with Theorem \ref{thm:HLL}, we have the following: 
\begin{corollary}\label{cor:CYleviflat}
Let $X$ be a compact Kähler manifold with $c_1(X)=0$ such that $\chi(X, \mathcal{O}_X)\not=0$. 
Assume that there exists a semi-positive holomorphic line bundle $L$ with ${\rm nd}(L)=1$. 
Then $X$ admits a real-analytic one-parameter family of Levi-flat hypersurfaces. 
\end{corollary}

\begin{proof}
By Theorem \ref{thm:HLL} and Kashiwara's theorem, 
we have a unitary flat line bundle $F$ and an effective divisor $D$ on $X$ such that $L^m$ is holomorphically isomorphic to $F\otimes [D]$ for some positive integer $m$. 
As $L^m$ is semi-positive, $[D]\cong L^m\otimes F^{-1}$ is also semi-positive. 
Therefore, by Theorem \ref{thm:main_gen}, one can take a covering $\{V_j\}$ of $Y:={\rm Supp}\,D$ by open subsets of $X$ and holomorphic functions $s_j\colon V_j \to \mathbb{C}$ whose zero sets coincide with $V_j\cap Y$, such that $|s_j/s_k|\equiv 1$ on each intersection $V_j\cap V_k$. 
The assertion then follows by considering the family of the hypersurfaces defined by $\textstyle\bigcup_j\{x\in V_j \mid |s_j(x)| = \varepsilon\}$ for $0<\ve\ll 1$. 
\end{proof}

Note that, for a Calabi--Yau manifold $X$, 
a widely expected conjecture states that any nef line bundle $L$ on $X$ is
numerically semi-ample; that is, there exists a semi-ample line bundle $L'$ on
$X$ such that $c_1^{\mathbb{R}}(L) = c_1^{\mathbb{R}}(L')$ 
(see, for example, \cite[Conjecture 1.1]{LOP} and \cite[Conjecture 3.8]{T}; see also \cite{FT1, FT2} and the explanation around \cite[Theorem~3.9]{T}, where it is shown that this conjecture fails if the line bundle is replaced with a $(1,1)$-class).  
If this conjecture holds, then Corollary \ref{cor:CYleviflat} would follow immediately.



\end{document}